\documentclass{amsart}
\usepackage{amsthm,amsmath,amssymb, graphics,enumerate}
\usepackage{tikz-cd}
\usepackage{mathrsfs}
\usepackage{adjustbox}
\usepackage{hyperref}

\newtheorem{thm}{Theorem}[section]
\newtheorem{claim}[thm]{Claim}
\newtheorem{cor}[thm]{Corollary}

\newtheorem{lem}[thm]{Lemma}
\newtheorem{prop}[thm]{Proposition}

\newtheorem{question}[thm]{Question}
\theoremstyle{definition}
\newtheorem{defn}[thm]{Definition}
\newtheorem{rem}[thm]{Remark}
\newtheorem{exa}[thm]{Example}
\newtheorem{notn}[thm]{Notation}

\newcommand\meet{\wedge}
\newcommand\join{\vee}
\newcommand\Zbb{\mathbb{Z}}

\newcommand\Seq[1]{( #1 )}

\newcommand\Acal{\mathcal{A}}
\newcommand\Bcal{\mathcal{B}}
\newcommand\Ccal{\mathcal{C}}
\newcommand\Dcal{\mathcal{D}}
\newcommand\Ecal{\mathcal{E}}
\newcommand\Fcal{\mathcal{F}}
\newcommand\Gcal{\mathcal{G}}
\newcommand\Hcal{\mathcal{H}}
\newcommand\Ical{\mathcal{I}}
\newcommand\Jcal{\mathcal{J}}
\newcommand\Qcal{\mathcal{Q}}

\newcommand\Gscr{\mathscr{G}}
\newcommand\Pscr{\mathscr{P}}
\newcommand\Sscr{\mathscr{S}}
\newcommand\Tscr{\mathscr{T}}

\renewcommand\Bbb{\mathbb{B}}
\newcommand\Hbb{\mathbb{H}}
\newcommand\Pbb{\mathbb{P}}
\newcommand\Nbb{\mathbb{N}}
\newcommand\Rbb{\mathbb{R}}
\newcommand\Th{{}^{th}}
\newcommand\add{\mathrm{add}}
\newcommand\dens{\mathrm{dens}}

\newcommand\Free{\mathrm{Free}}
\newcommand\colim{\mathrm{colim}}
\renewcommand\emptyset{\varnothing}

\newcommand\analytic{{\mathbf{\Sigma}}^1_1}
\newcommand\CA{{\mathbf{\Pi}}^1_1}
\newcommand\PCA{{\mathbf{\Sigma}}^1_2}
\newcommand\OCA{\mathrm{OCA}}

\newcommand\dfrak{\mathfrak{d}}

\newcommand\symdif{\triangle}
\newcommand\trleq{\trianglelefteq}

\newcommand\pbf{\mathbf{p}}

\newcommand\xbf{\mathbf{x}}
\newcommand\ybf{\mathbf{y}}
\newcommand\Abf{\mathbf{A}}
\newcommand\Bbf{\mathbf{B}}
\newcommand\Gbf{\mathbf{G}}
\newcommand\Xbf{\mathbf{X}}

\newcommand\bb[1]{^{[\![#1]\!]}}
\newcommand\id{\mathrm{id}}
\newcommand\im{\mathrm{im}}
\newcommand\alt{\mathrm{alt}}
\newcommand\sd{\mathrm{sd}}

\newcommand\PH{\mathrm{PH}}
\newcommand\dom{\operatorname{dom}}
\newcommand{\Om}{\Omega}

\newcommand\ON{\mathbf{ON}}

\newcommand\Coll{\mathrm{Coll}}
\begin{document}

\author[Bannister]{Nathaniel Bannister}
\author[Bergfalk]{Jeffrey Bergfalk}
\author[Moore]{Justin Tatch Moore}
\author[Todorcevic]{Stevo Todorcevic}

\keywords{
additivity,
coherent family,
derived limit,
Hechler forcing,
$L(\Rbb)$, 
measurable partition,
$n$-cofinal,
Solovay model,
strong homology,
universally Baire,
weakly compact
}

\thanks{
The authors would like to thank Jindrich Zapletal for his helpful remarks concerning the results
at the end of Section \ref{Sect:topology} and Monroe Eskew for discussion of the state of the art of dense ideals on small cardinals.
The authors would also like to thank the anonymous referee for their very careful reading of the paper.
The useful suggestions which they offered led to a number improvements.
The second author was partially supported by FWF grant Y1012-N35 and a Mar\'{i}a Zambrano Fellowship at the University of Barcelona.
The third author is supported by NSF grant DMS--1854367.
The fourth author is partially supported by CNRS grant UMR7586, NSERC grant 455916, and SFRS grant 7750027-SMART}

\address{Department of Mathematical Sciences, Carnegie Mellon University, USA}
\email{nathanib@andrew.cmu.edu}

\address{Departament de Matem\`{a}tiques i Inform\`{a}tica,
Universitat de Barcelona, Spain}
\email{bergfalk@ub.edu}

\address{Department of Mathematics, Cornell University, USA}
\email{justin@math.cornell.edu}

\address{Department of Mathematics, University of Toronto, Canada}
\email{stevo@math.toronto.edu}

\address{Institut de Math\'ematiques de Jussieu, Paris, France}
\email{stevo.todorcevic@imj-prg.fr}

\address{Matemati\v{c}ki Institut, SANU, Belgrade, Serbia}
\email{stevo.todorcevic@sanu.ac.rs}

\title{A descriptive approach to higher derived limits}

\begin{abstract}
We  present  a  new  aspect of the study of  higher  derived  limits.
More precisely,  we introduce a complexity measure for the elements of higher derived limits over the directed set $\Omega$ of functions from $\Nbb$ to $\Nbb$ and prove that cocycles of this complexity are images of cochains of the roughly the same complexity.
In the course of this work, we isolate a partition principle for powers of directed sets and show that whenever this principle holds, the corresponding derived limit ${\lim}^n$ is additive; vanishing results for this limit are the typical corollary.
The formulation of this partition hypothesis synthesizes and
clarifies several recent advances in this area.
\end{abstract}

\maketitle

\section{Introduction}

The first explicit treatments of the derived limits ${\lim}^n$
of the inverse limit functor were a cluster of works appearing around 1960
\cite{EiMo,Mil,Nob,Roos,Yeh}.\footnote{Higher derived limits did, however, implicitly figure in earlier works by, most notably, both Steenrod \cite{Ste} and Cartan and Eilenberg \cite{CaEi}.}
Milnor's \cite{Mil} may have been the most influential, for his isolation therein of a ${\lim}^1$ term in the cohomology of a mapping telescope foreshadowed this functor's role in a wide range of (co)limit phenomena of greater later prominence in algebraic topology---in the study of localizations and completions, homotopy (co)limits, and phantom maps, for example \cite{BK, MaPo, McG}.
Eilenberg and Moore established the first derived limit's importance for spectral sequence computations at around the same time \cite{EiMo}, and connections of the higher derived limits ${\lim}^n$ to the homological dimensions of rings were discovered very soon thereafter \cite{Oso}.
They have formed a fundamental part of mathematicians' toolkits ever since \cite{Wei}.

Within just a few years, the relevance of set-theoretic considerations to derived limits had grown conspicuous \cite{Gob, Mitch, Oso}.
This relevance has manifested in more recent decades as a growing literature on the set theory of the derived limits of inverse systems indexed by the partial order ${^\omega}\omega$ (see \cite[\textsection 1]{SVHDL} for a brief survey).
A major stimulus for much of this work was the formulation in \cite{MP} of a necessary condition for the additivity of strong homology in terms of such limits, together with the reformulation of the most basic instance of this condition in terms of the \emph{triviality} of certain \emph{coherent families of functions} indexed by ${^\omega}\omega$.
This last condition was promptly shown to be independent of the ZFC axioms \cite{MP, DSV, To89}, leaving open for the next three decades the question of whether the additivity of strong homology on any robust class of topological spaces might be independent of the ZFC axioms as well.

This question was answered in \cite{strong_hom_add} by developing the aforementioned implications into a more general circuit of equivalences, or near-equivalences, between
\begin{enumerate}
\item \label{item1} the commutativity of the ${\lim}^n$ and $\colim$ functors,
\item \label{item2} the additivity of strong homology, and
\item \label{item3} the triviality of higher-dimensional coherent families of functions,
\end{enumerate}
each on suitably restricted domains.
The third of these items had recently been shown to be consistent relative to a weakly compact cardinal in \cite{SVHDL}; it followed that item (\ref{item2}) on the category of locally compact separable metric spaces and, equivalently, item (\ref{item1}) for countable discrete diagrams of inverse sequences of finitely-generated abelian groups (in the category of pro-abelian groups) were both consistent relative to a weakly compact cardinal as well.
The solution, in short, consisted of a reduction of the questions of (\ref{item1}) and (\ref{item2}) to the more combinatorial question of (\ref{item3}), together with a solution to the latter.

The present work takes this reduction one step further: here we introduce a family of purely set-theoretic hypotheses $\PH_n$ on partitions of powers of ${^\omega}\omega$ and show these hypotheses to lie at the heart of the rather technical arguments of \cite{strong_hom_add} and \cite{SVHDL}.
More precisely, we distill those arguments into an initial step establishing the simultaneous consistency of the principles $\PH_n$ for all $n\in\omega$,
followed by a ZFC deduction from the latter of conditions (\ref{item1}) through (\ref{item3}) above.

This decomposition carries a number of benefits.
To begin with, it fully disentangles the combinatorial and algebraic components of arguments whose hybridity had hitherto been an impediment to their comprehension.
It thereby facilitates a much closer analysis of the set-theoretic content of these arguments, and this, indeed, is our work's main contribution: we show that, in the presence of a suitable large cardinal assumption,
the partition hypotheses $\PH_n$ hold for all \emph{universally Baire} partitions of powers of ${^\omega}\omega$.
This carries the corollary that, modulo a mild large cardinal hypothesis, for all $n\in\omega$ all universally Baire $n$-coherent families of functions indexed by ${^\omega}\omega$ are trivial,
answering Questions 6 and 7.12 of \cite{B17} and \cite{SVHDLwLC}, respectively, and generalizing a main result of \cite{To98}.
Moreover, the associated trivializations can themselves be taken to have low complexity relative to the $n$-coherent family; in particular, in the presence of suitable large cardinal hypotheses, they are universally Baire.

The partition hypotheses $\PH_n$, moreover, are clearly of some interest in their own right.
For example, they readily generalize to any directed partial order $\Lambda$ and to the ordinals $\omega_n$ in particular.
In our penultimate section, we record several basic but intriguing recognitions about the latter.
We show, for example, that the hypothesis $\PH_n(\omega_n)$ fails for all $n\in\omega$.
We also describe, for each $n$, conditions implying the principle $\PH_n(\omega_{n+1})$.
Already when $n=1$, however, these conditions carry considerable large cardinal strength, while for $n>1$ they are not even known to be consistent with the ZFC axioms.

\subsection*{Organization of this paper}
Section \ref{Sect:prelim} contains a review of standard notation, definitions, and set-theoretic concepts which will be used throughout the paper.
It also contains a list of references which give more detailed introductions to the different topics in set theory and algebraic topology which we will need.
In Section \ref{Sect:PH}, we formulate the Partition Hypotheses $\PH_n$ which are our main object of study.
Sections \ref{Sect:topology} and \ref{Sect:MeasPartHyp} together comprise our descriptive set-theoretic analysis of these principles.
In Section \ref{Sect:topology} we introduce refinements of the standard topological and Baire measurability structures on powers of ${^\omega}\omega$ and in Section \ref{Sect:MeasPartHyp} we show that the hypotheses $\PH_n$ hold for partitions which are measurable with respect to these structures.
We obtain as an immediate corollary a negative answer (modulo a large cardinal hypothesis) to the question, appearing in both \cite{B17} and \cite{SVHDLwLC}, of whether a nontrivial $n$-coherent family may be analytic.

Sections \ref{Sect:forcing} and \ref{Sect:PartHypAdd} together show how the hypotheses $\PH_n$ effect a decomposition of recent consistency results on the additivity of strong homology, and of ${\lim}^n$, into two distinct steps.
Section \ref{Sect:forcing} shows that these hypotheses hold in any generic extension by a finite-support iteration of Hechler forcings of weakly compact length.
This is the only explicit appearance of forcing in our arguments; readers who are not fluent with this technique may take Theorem \ref{PHfromHechlers} as a black box at no cost to their comprehension of any other section of the paper.
Section \ref{Sect:PartHypAdd} deduces additivity conclusions for ${\lim}^n$ from $\PH_n$. 
Section \ref{Sect:Generalizing} treats the two most natural generalizations of these hypotheses, namely to arbitrary products of directed posets, and to the ordinals.
After showing that the first of these carries the additivity implications for higher limits that one would hope for, we record the results on the ordinals $\omega_n$ mentioned above; we then reconnect these results to our main focus, partition hypotheses on powers of ${^\omega}\omega$.
In a third subsection, we show that partition hypotheses admit succinct formulation within the framework of simplicial sets, formulations in which they figure, suggestively, as only minor variations on classical partition relations.
We conclude with a number of open questions.

\section{Notation and preliminaries}
\label{Sect:prelim}

Although this paper is intended to be self-contained, we begin by listing
some standard references which some readers may find helpful, depending on their
background.
General information about set theory, including forcing, can be found in
Kunen's \cite{set_theory:kunen}.
Kechris's \cite{kechris} is the standard reference for descriptive set theory (e.g. \emph{Borel}, \emph{$\analytic$-sets}, \emph{$\CA$-sets}).
Kanamori's \cite{higher_infinite} is an encyclopedic account of large cardinals (e.g. \emph{weakly compact}, \emph{supercompact}, \emph{measurable}, $x^\sharp$), including their
history and motivation.
It also includes some additional results in descriptive set theory which will be needed (\emph{Shoenfield absoluteness}, \emph{Martin-Solovay absoluteness}).
An introduction to the notion of a \emph{universally Baire set} can be found in \cite{UB}.
Relevant background material on homological algebra and on higher derived limits
can be found in Marde\v{s}i\'c's \cite{SSH}.
Finally, many of the proofs in the present paper have their roots in
\cite{strong_hom_add, SVHDL}.

We now turn to our review.
As is standard, the symbol $\omega$ will denote the set of finite ordinals,
a set which coincides with the nonnegative integers.
All counting and indexing will begin at $0$ unless otherwise indicated.
If $X$ is a set and $n \in \omega$, we will write $X^n$ to denote the set of $n$-tuples of elements of $X$.
We will identify $X^1$ with $X$ and $X^n \times X$ with $X^{n+1}$.
Also, if $f$ is a function defined on a subset of $X^n$ and $(x_0,\ldots,x_{n-1})$ is in the domain of $f$,
we will write $f(x_0,\ldots,x_{n-1})$ instead of $f((x_0,\ldots,x_{n-1}))$.
In particular, if $f$ is a function defined on a subset of $X^{n+1}$,
$\xbf=(x_0,\ldots,x_{n-1}) \in X^n$, and $y \in X$,
we will write $f(\xbf,y)$ for $f(x_0,\ldots,x_{n-1},y)$ which in turn is really
$f((x_0,\ldots,x_{n-1},y))$.
By convention the $0$-tuple, also known as the null sequence,
is the empty set $\emptyset$.

We will use $\Omega$ to denote the collection of all strictly
increasing functions from $\omega$ to $\omega$ and $\Sigma$ to denote all finite strictly increasing sequences
of elements of $\omega$.
If $s \in \Sigma$, let $[s]$ denote the set of all elements of $\Omega$ which extend $s$.
This is a basic clopen set in the Polish topology on $\Omega$.
If $x,y \in \Omega$, we will write $x \join y$ and $x \meet y$ for their coordinatewise maximum and
minimum respectively.
\emph{Borel} will always refer to the metric topology on $\Omega$ and its powers unless specified otherwise.

Recall that a \emph{quasi-order} is a set $\Pbb$ equipped with a reflexive
transitive relation; we will also write \emph{quasi-order} to refer to the relation itself.
We will say that $\Pbb$ is a \emph{quasi-lattice} if there exist associative
operations $\meet$ and $\join$ on $\Pbb$
such that:
\begin{itemize}

\item $x \meet y \leq x,y$ and for all $z$, we have that $z \leq x,y$ implies $z \leq x \meet y$;

\item $x,y \leq x \join y$ and for all $z$, we have that $x,y \leq z$ implies $x \join y \leq z$.

\end{itemize}

If $\Pbb$ is a quasi-order and $n \geq 0$,
$\Pbb^{[n]}$ consists of those elements of $\Pbb^n$ whose coordinates
occur in (weakly) increasing order ($\Pbb^{[0]} = \Pbb^0$ consists
only of the null sequence).
We will write $\Pbb^{\leq n}$ to denote $\bigcup_{i=1}^n \Pbb^{i}$
and $\Pbb^{[\leq n]}$ to denote $\bigcup_{i=1}^n \Pbb^{[i]}$
(note that $i=0$ is excluded in both cases).
If $\xbf,\ybf \in \Pbb^{\leq n}$, we will write $\xbf \trleq \ybf$ to denote that $\xbf$ can be obtained from $\ybf$ by deleting coordinates. 

Let $\Pbb\bb{n} \subseteq \prod_{i=1}^n \Pbb^{i}$ consist of all $\sigma$ which are
$\trleq$-increasing.
If $F:\Omega^{\leq n} \to \Omega$ is such that $\xbf \trleq \ybf$ implies $F(\xbf) \leq F(\ybf)$,
then define $F^*:\Omega\bb{n} \to \Omega^{[n]}$
by $$F^*(\sigma) := F\circ \sigma = \Seq{F(\sigma(i)) \mid 1 \leq i \leq n}.$$
We note here that by our convention, both $\Pbb\bb{1}$ and $\Pbb^{[1]}$ are identified with $\Pbb$.

The quasi-order of primary interest in this paper is $\Omega$ equipped with the 
order of eventual dominance.
For all $k \in \omega$ we equip $\Omega \cup \Sigma$ with the quasi-order
$\leq^k$ defined by $x \leq^k y$ if $\dom(y) \subseteq \dom(x)$ and
$x(i) \leq y(i)$ for all $i \in \dom (y)$ with $k \leq i$.
If $x,y \in \Omega$, we define $x \leq^* y$ to mean $x \leq^k y$ for some $k$.
We will write $\leq$ for $\leq^0$, noting that in this case
$\leq$ is antisymmetric and hence a partial order.
If $k > 0$, both $\leq^k$ and $\leq^*$ fail to be antisymmetric.
Notice that the operations $x \meet y := \min (x,y)$ and $x \join y := \max(x,y)$ witness
that $(\Omega,\leq^k)$ is a quasi-lattice for each $k$ and that $(\Omega,\leq^*)$ is a quasi-lattice.
If an implicit reference is made to a quasi-order on $\Omega$, it refers to
$\leq^*$.
For instance, 
$\Omega^{[n]}$ consists of tuples which are $\leq^*$-increasing.

Section \ref{Sect:forcing} assumes the reader is proficient in forcing.
As noted, a standard treatment of forcing can be found in \cite{set_theory:kunen};
we will also utilize the ``dot convention'' for denoting names for elements of generic extensions---see, e.g., \cite{Jech}.
Recall that \emph{Hechler forcing} is the partially ordered
set $\Hbb$ consisting of all pairs $p=(s_p,x_p)\in\Sigma\times\Omega$ for which $s_p\subseteq x_p$.
We call $s_p$ the \emph{stem} of $x_p$.
The order on $\Hbb$ is defined by
$q\leq p$ if and only if $s_p\subseteq s_q$ and $x_q\geq x_p$.
Following standard forcing terminology,
elements of $\Hbb$ will be referred to as
\emph{Hechler conditions}
or simply \emph{conditions}.
It should be noted that while only Section \ref{Sect:forcing} will require knowledge of forcing,
$(\Hbb,\leq)$ is closely tied to the topology $\tau$ on $\Omega$
appearing in Section \ref{Sect:topology}.
Moreover the elementary definitions and arguments of Section \ref{Sect:topology} are informed
by the perspective of forcing.

Recall that a \emph{Polish space} is a topological
space which is separable and completely metrizable.
The class of Polish spaces is closed under countable products and taking closed subspaces.
In particular, $\Hbb \subseteq \Sigma \times \omega^\omega$ is a Polish space, where $\Sigma$ is equipped with the discrete topology.
Also, if $X$ is a countable set, then $\Pscr(X)$ is a compact Polish space when given the topology generated by the following sets, for $x \in X$:
$$\{A \in \Pscr(X) : x \in A\} \qquad \qquad \{A \in \Pscr(X) \mid x \not \in A\}.$$
The Borel sets in a Polish space are the elements of the smallest $\sigma$-algebra which
contains the open sets.
The projective hierarchy of $\mathbf{\Sigma}^1_n$- and $\mathbf{\Pi}^1_n$-sets is defined recursively as follows.
The $\mathbf{\Sigma}^1_0$-sets are the Borel sets.
The $\mathbf{\Pi}^1_n$-sets are those sets whose complement is $\mathbf{\Sigma}^1_n$.
The $\mathbf{\Sigma}^1_{n+1}$-sets are those sets which are a continuous image (e.g. a projection) of a $\mathbf{\Pi}^1_n$-set.
$\analytic$-sets are often referred to as \emph{analytic sets};
$\CA$-sets are often referred to as \emph{coanalytic sets}.
By a result of Souslin, the Borel sets are precisely those sets which are both anaytic and coanalytic.
A function is Borel (or $\CA$, $\PCA$) if its graph is.
We note that the $\mathbf{\Sigma}^1_{n+1}$-functions include the $\mathbf{\Pi}^1_n$-functions and are closed under composition.
Borel functions coincide with those functions with the property that preimages of open sets are Borel.

In order to demonstrate that a set is projective, one typically examines the logical structure of the description of the set.
For instance Borel sets are closed under quantification over countable sets:
if $X$ is Polish, $S$ is countable, and $B \subseteq X \times S$ is Borel, then so are the sets:
\[
\{x \in X \mid \exists s \in S ((x,s) \in B)\}
\]
\[
\{x \in X \mid \forall s \in S ((x,s) \in B)\}.
\]
Similarly, if $n \geq 1$,
the $\mathbf{\Sigma}^1_n$-sets are closed under existential quantification over a Polish space while the $\mathbf{\Pi}^1_n$-sets are closed under
universal quantification over a Polish space.
Furthermore, if $Z \subseteq X \times Y$, $X$ and $Y$ are Polish and $Z$ is $\mathbf{\Pi}^1_n$,
then
\[
\{x \in X \mid \exists y \in Y ((x,y) \in Z)\}
\]
is a $\mathbf{\Sigma}^1_{n+1}$-set.

At several points in our argument, we will need the
Kond\^{o}-Novikov Uniformization Theorem for $\CA$-sets.

\begin{thm}(see \cite[36.14]{kechris}) \label{KNU}
  Suppose that $X$ and $Y$ are Polish spaces and $R \subseteq X \times Y$ is a $\CA$-relation.
  There is $\CA$-function $\varphi \subseteq R$ with the same domain as $R$.
\end{thm}

Recall that a subset $A$ of a topological space $X$ has \emph{the property of Baire} in $X$
if there is an open subset $U$ of $X$ such that the symmetric difference between $A$
and $U$ is a meager subset of $X$.

\begin{defn}
A subset $A$ of a Polish space $X$ is \emph{universally Baire}
if for every Hausdorff topological space $Y$ and every continuous map $f: Y\rightarrow X$, the preimage $f^{-1}(A)$ has the property of Baire in $Y$.
A function $f:A \to B$ is universally Baire if $A$ and $B$ are universally Baire subsets of Polish spaces
and the graph of $f$ is universally Baire.
\end{defn}

The collection of subsets of a given Polish space which are universally Baire forms a $\sigma$-algebra.
Moreover, all $\analytic$- and $\CA$-sets are universally Baire.
In the presence of large cardinals, the universally Baire sets enjoy much stronger closure properties.

\begin{thm} \cite{UB} \cite{Omega-conjecture} \label{rel_UB}
Suppose either that there is a supercompact cardinal or a proper class of Woodin cardinals.
If $X,Y \in L(\Rbb)$ are Polish spaces, $A \subseteq X$ is universally Baire, and $B \subseteq Y$ is in 
$L(\Rbb,A)$, then $B$ is universally Baire.
In particular, every $\PCA$-set is universally Baire
and the class of universally Baire functions is closed
under composition.
\end{thm}

Here $L(\Rbb)$ is the minimum model of ZF which contains all of the reals and ordinals;
if $A \subseteq \Rbb$, $L(\Rbb,A)$ is the minimum model of ZF which contains all of the reals, ordinals, and the set $A$.
Notice that the assumption that the Polish spaces $X$ and $Y$ are in $L(\Rbb)$ is a superficial one:
any Polish space
is homeomorphic to a closed subspace of $\Rbb^\Nbb$ and $L(\Rbb)$ contains all such subspaces.

\section{A Partition Hypothesis for \texorpdfstring{$\Omega^{n}$}{On}}
\label{Sect:PH}

We now define our main object of study.

\begin{defn}
\label{def:main_n_cofinal}
Suppose that $\Pbb$ is a directed quasi-order.
A function $F:\Pbb^{\leq n} \to \Pbb$ is \emph{$n$-cofinal} if:
\begin{itemize}

\item $x \leq F(x)$ for all $x \in \Pbb$;

\item if $\xbf \trleq \ybf$ are in $\Pbb^{\leq n}$, then $F(\xbf) \leq F(\ybf)$. 

\end{itemize}
\end{defn}

\begin{defn}
For all $n \in \omega$, define the \emph{Partition Hypothesis} associated to $\Pbb^{n+1}$ and a cardinal $\lambda$ to be the following statement:
\begin{description}

\item[$\PH_n (\Pbb,\lambda)$]
For all $c:\Pbb^{n+1} \to \lambda$ there is an $(n+1)$-cofinal
$F:\Pbb^{\leq n+1} \to \Omega$ such that $c \circ F^*$ is constant. 

\end{description}
$\PH_n$ will denote $\PH_n(\Omega,\omega)$.
\end{defn}
Our interest will be exclusively in the special case $\Pbb = \Omega$ and $\lambda = \omega$ for much of the paper; 
we will return to the general setting in Section \ref{Sect:Generalizing}.

Several observations are now in order.
First, the values of $c$ on elements of $\Omega^{n+1}$ which are not
$\leq^*$-increasing are not relevant---i.e., $\PH_n$ is really a statement about partitions of
$\Omega^{[n+1]}$.
Second, notice that partition hypotheses grow in strength with $n$, in the sense that $\PH_{n+1}$ implies $\PH_n$ for all $n\in\omega$.
In Section \ref{Sect:Generalizing} we record partition hypotheses of order $n$ which are consistent, but whose order-$(n+1)$ instances are not.
Third, $\PH_0$ is a ZFC theorem.
In fact, something formally stronger is true:
since the $\leq^{*}$ ordering on $\Omega$ is $\sigma$-directed,
for any $f:\Omega \to\omega$ there exists an $i\in\omega$ such that $\Upsilon:= f^{-1}(i)$ is $\leq^{*}$-cofinal and consequently $\leq^k$-cofinal in $\Omega$ for some $k\in\omega$
(see \cite[Lemma 3]{strong_hom_add}).
Let $F:\Omega \to \Upsilon$ be such that $x \leq^k F(x)$ for all $x \in \Omega$; 
this $F$ then witnesses $\mathrm{PH}_0$ not only with respect to the $\leq^*$ ordering, but with respect to the $\leq^k$ ordering on $\Omega$ as well.

The principle at work here implies the following more general lemma.

\begin{lem}\label{extensionlemma} Fix $n\in\omega$ and a function $c:\Omega^{n+1}\to\omega$.
For any $\leq^*$-cofinal $\Upsilon\subseteq\Omega$ and $(n+1)$-cofinal function $F:\Upsilon^{\leq n+1}\to\Omega$ extending the identity map for which the composition $c\circ F^*$ is constant, there exists an $(n+1)$-cofinal $\bar{F}:\Omega^{\leq n+1}\to\Omega$ extending $F$ with $c\circ \bar{F}^*$ constant.
Moreover, if there is a $k$ such that $F$ is $(n+1)$-cofinal with respect to $\leq^k$ then for some $\ell\in\omega$, $\bar{F}$ may be taken to $(n+1)$-cofinal with respect to the ordering $\leq^\ell$; similarly, if $F$ maps into $\Upsilon$ then $\bar{F}$ may be taken to map into $\Upsilon$ as well.
\end{lem}
\begin{proof}
Suppose that $n$, $F$, and $\Upsilon$ are given as in the statement of the lemma and that $F$ is $(n+1)$-cofinal with respect to $\leq^k$. 
Again observe that $\Upsilon$ is $\leq^m$-cofinal in $\Omega$ for some $m\in\omega$.
Let $\ell=\max(k,m)$. 
Fix a function $g:\Omega\to\Upsilon$ with $g\restriction\Upsilon=\id$ and $x\leq^\ell g(x)$ for all $x\in\Omega$.
Define $\bar{F}$ by setting $\bar{F}(x_0,\dots,x_j):= F(g(x_0),\dots,g(x_j))$.
It is easily checked that $\bar{F}$ is an $(n+1)$-cofinal function as in the lemma's conclusion.
\end{proof}

This ``concentration on some $\leq^k$'' phenomenon holds for our partition hypotheses very generally.
More precisely, note that by modifying $c$ in the statement of $\PH_n$, we may assume
that whenever $\xbf$ is in $\Omega^{[n]}$, $c(\xbf)$ records (in addition to its original data) the least $k$
such that the coordinates of $\xbf$ are $\leq^k$-increasing.
If $F$ is an $(n+1)$-cofinal function such that $c \circ F^*$ is constant, then
for some $k$, $F(\xbf)\leq^kF(\ybf)$ for each $\xbf \trianglelefteq \ybf$. 
Moreover, since $(\Om,\leq^*)$ is $\sigma$-directed, there is an $\ell$ such that $\Upsilon_\ell:=\{x\mid x\leq^\ell F(x)\}$ is $\leq^*$ cofinal and hence $\leq^m$ cofinal for some $m$. 
Extending $F\upharpoonright\Upsilon_\ell$ to all of $\Omega^{[n]}$ as in Lemma \ref{extensionlemma} yields an $\bar{F}$ which is $(n+1)$-cofinal with respect to the order $\leq^{\max(m,k,\ell)}$.
In consequence, in any application of $\PH_n$ below, we may assume that the witnessing function $F$ is $(n+1)$-cofinal with respect to $\leq^k$ for some $k$.

\section{A notion of measurability associated to \texorpdfstring{$\Omega^{[n]}$}{O[n]}}

\label{Sect:topology}
In this section, we introduce variants of the standard topological and Baire measurability structures on $\Omega^{[n]}$ which will be instrumental in the argument of our main results.
\subsection*{A topology on \texorpdfstring{$\Omega$}{O}}

In addition to the Polish topology on $\Omega$,
we will also utilize the following stronger topology $\tau$
which is generated by the basic open sets
$$N_k(x) := \{y \in \Omega \mid x \leq y \textrm{ and } x \restriction k = y \restriction k \}.$$
This topology (already considered in \cite{To89})  is first countable, Choquet, and
has a $\sigma$-centered base (although it is nonseparable).
Notice that there is a natural order isomorphism between the Hechler poset $\Hbb$ and the basic open sets in $\tau$ ordered by containment:
$p \mapsto N_{|s_p|}(x_p)$;
we will let $N_p$ denote $N_{|s_p|}(x_p)$.
If $\pbf \in \Hbb^\omega$, define $W(\pbf) = \bigcup_{n=0}^\infty N_{\pbf(n)}$.
Clearly $W(\pbf)$ is Borel and $\tau$-open.
We will now isolate a sufficient criterion on $\pbf \in \Hbb^\omega$ to ensure that $W(\pbf)$ is dense.

Define $S:\Hbb^\omega \to \Pscr(\Sigma)$ by $S(\pbf) := \{s_{\pbf(n)} \mid n \in \omega\}$ and note
that $S$ is a Borel function---its graph is a Borel subset of the Polish space
$\Hbb^\omega \times \Pscr(\Sigma)$.
To see this, observe that
$(\pbf,A)$ is in the graph of $S$ if and only if
\[
(\forall n \in \omega\ (s_{\pbf(n)} \in A)) \land (\forall s \in \Sigma\ \exists n \in \omega\ ((s \not \in A) \lor (s = \pbf(n)))).
\]
Since the set of triples $(\pbf,A,n)$ such that $s_{\pbf(n)} \in A$ is clopen and hence Borel, the set of $(\pbf,A)$ such that 
$\forall n \in \omega\ (s_{\pbf(n)} \in A))$ is Borel (in fact it is closed).
Similarly, the set of pairs $(\pbf,A)$ such that 
\[
\forall s \in \Sigma\ \exists n \in \omega\ ((s \not \in A) \lor (s = \pbf(n)))
\]
is Borel (in fact $G_\delta$).
It follows that the graph of $S$ is Borel.
(This is an illustration of complexity computation mentioned in Section \ref{Sect:prelim}.)

\begin{defn}
A subset $S_0$ of $\Sigma$ is \emph{strongly dense}
if for all $(s,x) \in \Hbb$,
there is a $t \in S_0$ extending $s$ such that $x \leq t$.
We will let $\Sscr$ denote the collection of all strongly dense subsets of $\Sigma$.
\end{defn}

\begin{lem} \label{dense_open_basis}
If $\pbf \in \Hbb^\omega$ and $S(\pbf)$ is strongly dense, then $W(\pbf)$ is dense.
Moreover, if $S_0 \subseteq \Sigma$ is strongly dense and $U \subseteq \Omega$ is dense and open 
with respect to $\tau$, then
there is a $\pbf \in \Hbb^\omega$ such that $S(\pbf) \subseteq S_0$, $W(\pbf) \subseteq U$,
and $S(\pbf)$ is strongly dense.
\end{lem}

\begin{proof}
First suppose that $S(\pbf)$ is strongly dense for some $\pbf \in \Hbb^\omega$.
Let $N_k(x)$ be an arbitrary basic open set.
Since $S(\pbf)$ is strongly dense, there is an $n \in \omega$ such that $s_{\pbf(n)}$
extends $x \restriction k$ and $x \leq s_{\pbf(n)}$.
Define $y = x \join x_{\pbf(n)}$ and observe that
$y \in N_k(x) \cap W(\pbf)$ with $y \in W(\pbf)$ witnessed by $n$.

Now suppose that $U$ and $S_0$ are given as in the statement of the lemma.
Let $S_1$ be the set of all $s \in S_0$ such that for some $x \in \Omega$ extending $s$,
$N_{|s|}(x) \subseteq U$.
Let $\pbf \in \Hbb^\omega$ be such that $S(\pbf) = S_1$ and such that if $(s,x)$ is in the range of $\pbf$,
then $N_{|s|}(x) \subseteq U$.
Since $S(\pbf) \subseteq S_0$ and $W(\pbf) \subseteq U$, it suffices to show that
$S_1$ is strongly dense.

To this end, let $q \in \Hbb$ be arbitrary.
Since $U \cap N_q$ is nonempty, it contains a basic open set $N_{q'}$.
Since $S_0$ is strongly dense, there is a $t \in S_0$ which extends $s_{q'}$ such
that $x_{q'} \leq t$.
If $y \in \Omega$ extends $t$ and $x_{q'} \leq y$, then $N_{|t|}(y) \subseteq U$ and therefore
$t \in S_1$, $t$ extends $s_{q}$, and $x_q \leq t$ as desired.
\end{proof}

Define ${\widehat \Gscr} := (\Hbb^\omega)^\omega$ and let 
$\Gscr \subseteq {\widehat \Gscr}$ to be the set of sequences
$G$ such that for all $n \in \omega$, $S(G(n)) \in \Sscr$.
If $G \in {\widehat \Gscr}$, define $W(G):=\bigcap_{n=0}^\infty W(G(n))$;
note that $W(G)$ is a Borel set which is $G_\delta$ with respect to $\tau$.
Furthermore, if $G \in \Gscr$, then $W(G)$ is dense.
We now turn to some complexity calculations.

\begin{lem} \label{Gscr_CA}
$\Sscr$ and $\Gscr$ are $\CA$-sets and there are $\CA$-functions
$g:\Gscr \times \Hbb \to \Omega$ and $\tilde g:\Gscr^{<\omega} \times \Hbb \to \Omega$ such that
$$g(G,p) \in W(G) \cap N_p$$
$$\tilde g(G_0,\ldots,G_{n-1},p) \in \bigcap_{i < n} W(G_i) \cap N_p$$
whenever $G \in \Gscr$, $(G_0,\ldots,G_{n-1}) \in \Gscr^{<\omega}$, and $p \in \Hbb$.
\end{lem}

\begin{proof}
Define $\Tscr$ to be the set of all $(S_0,s,x) \in \Pscr(\Sigma) \times \Sigma \times \Omega$
such that there is a $t \in S_0$ which extends $s$ with $x \leq^{|s|} t$.
Since $\Tscr$ is Borel (it is metrically open), it follows that
$$
\Sscr = \{S_0 \in \Pscr(\Sigma) \mid \forall (s,x) \in \Sigma \times \Omega \ ((S_0,s,x) \in \Tscr)\}
$$
is $\CA$ ($\Sscr$ is the complement of the projection of the complement of $\Tscr$; we can also see this by appealing to the fact that the $\CA$-sets include the Borel
sets and are closed under universal quantification over Polish spaces).
Since the $\CA$-sets are closed under taking preimages by Borel functions \cite[32A]{kechris},
it follows that the preimage of $\Sscr$ under $S$ is $\CA$.
Since $\Gscr \subseteq (\Hbb^\omega)^\omega$ consists of those
sequences $\Seq{\pbf_n \mid n \in \omega}$ such that for all $n$, $S(\pbf_n) \in \Sscr$,
$\Gscr$ is an intersection of countably many $\CA$-sets and hence is $\CA$ (see \cite[32A]{kechris}).

Now consider the relation $R \subseteq {\widehat \Gscr} \times \Hbb \times \Omega$ consisting of 
all $(G,p,y)$ such that $G \in \Gscr$ and $y \in W(G) \cap N_p$.
Since the set of all $(G,p,y)$ such that $y \in W(G) \cap N_p$ is a Borel set, $R$ is $\CA$.
By Theorem \ref{KNU}, there is a $\CA$-function
$g:\Gscr \times \Hbb \to \Omega$ such that the graph of $g$ is contained in $R$.
Now if $G = (G_0,\ldots,G_n) \in \Gscr^{<\omega}$, define $\widetilde G \in \Gscr$ by
$\widetilde G(k) = G_m (i)$ if $k = m n + i$; 
if $G$ is the null sequence, define $\widetilde G$ to be the contant sequence with value the greatest element of $\Hbb$
(corresponding to the trivial open set $\Omega$).
Define $\tilde g(G,p) = g({\widetilde G},p)$, noting that $\tilde g$ is also a $\CA$-function.
\end{proof}

We will fix, for the remainder of the paper, $\CA$-functions $g$ and $\tilde g$ satisfying the conclusion of Lemma \ref{Gscr_CA}.

\subsection*{\texorpdfstring{$\Hcal_n$}{Hn}-measurability}

We will now develop an abstract higher dimensional analog of Baire measurability with respect to $\tau$ which
we will call \emph{$\Hcal_n$-measurability}.
We will prove that, in the presence of a large cardinal hypothesis,
the Partition Hypothesis holds for $\Hcal_n$-measurable partitions of $\Omega^{[n]}$ 
and that \emph{universally Baire} subsets of $\Omega^{[n]}$ are $\Hcal_n$-measurable. 
While it seems possible to show that $\Hcal_n$-measurability is Baire measurability with respect to a suitable topology
on $\Omega^{[n]}$, this would introduce unnecessary
complications and we choose not to pursue this.
We note that, if one is willing to make a stronger large cardinal assumption, it is possible
to prove the results of this section using the general framework provided by \cite[\S 5.1]{forcing_idealized}.
 
Define $\Hcal$ to be the collection of all Borel sets of the form $N_k(x) \setminus E$ such that
$E$ is $\tau$-meager.
The \emph{stem} of $N_k(x) \setminus E$ is $x \restriction k$.
This is well defined since every nonempty $\tau$-open set is nonmeager and since
the basic open sets $N_k(x)$ are both closed and open.
If $n \geq 0$, $Z \subseteq \Omega^{[n+1]}$, and $\xbf = (x_0,\ldots,x_{n-1}) \in \Omega^{[n]}$, define
$$Z_\xbf := \{y \in \Omega \mid (x_0,\ldots,x_{n-1},y) \in Z\}.$$
Note that if $n=0$, then $\Omega^{[0]} = \{\emptyset\}$ and $Z_\emptyset = Z$ modulo our convention of identifying $Z$ and $Z^1$.
We will now recursively define $\Hcal_n$ for $n \geq 0$ as well as define what the stem of an element of
$\Hcal_n$ is.

\begin{defn}
Set $\Hcal_0 = \{\Omega^{[0]}\}$; the stem of $\Omega^{[0]}$ is $\emptyset$.
Define $\Hcal_{n+1}$ to consist of all Borel sets $Z \subseteq \Omega^{[n+1]}$ such that for some $s \in \Sigma$:
\begin{itemize}

\item $X:=\{\xbf \in \Omega^{[n]} \mid Z_\xbf \ne \emptyset\}$ is in $\Hcal_n$;

\item for all $\xbf \in \Omega^{[n]}$ either $Z_\xbf \in \Hcal$ with stem $s$ or $Z_\xbf = \emptyset$.

\end{itemize}
The stem of $Z$ is the element of $\Sigma^{n+1}$ whose first $n$ entries are the stem of $X$ and whose
final entry is $s$.
\end{defn}

Observe that if $A \subseteq B$ are in $\Hcal_n$, then the stem of $A$ extends the stem of $B$ coordinatewise.
Also, $\Hcal_1$ coincides with $\Hcal$ modulo our convention of identifying $\Omega^{[1]}$ and $\Omega$.

There are two natural notions of smallness associated to each $\Hcal_n$.
We will ultimately show that they coincide if we assume a large cardinal hypothesis.

\begin{defn}
  A subset $X$ of $\Omega^{[n]}$ is \emph{$\Hcal_n$-nowhere dense} if for every $A \in \Hcal_n$ there is a $B \subseteq A \setminus X$ in $\Hcal_n$;
  $X$ is \emph{$\Hcal_n$-meager} if it is a countable union of $\Hcal_n$-nowhere dense sets.
  Similarly one defines $\Hcal$-nowhere dense and $\Hcal$-meager.
\end{defn}

\begin{defn}
Define  $\Ical$ to be the $\sigma$-ideal of $\tau$-meager subsets of $\Omega$ and
define  $\Ical_0:= \{\emptyset\}$ to be the trivial ideal on $\Omega^{[0]}$.
Define $\Ical_{n+1}$ to consist of all $I \subseteq \Omega^{[n+1]}$ such that for some Borel set
$Z \subseteq \Omega^{[n+1]}$, $I\subseteq Z$ and
\[
\{\xbf \in \Omega^{[n]} \mid Z_\xbf \not \in \Ical\} \in \Ical_n
\]
\end{defn}

It is easily verified that each $\Ical_n$ is closed under taking countable unions.
It will be useful to work with certain elements of $\Ical_n$ having a particularly nice form.

\begin{defn}
A set $I$ in $\Ical_{n+1}$ is \emph{full}  if:
\begin{itemize}

\item  $I$ is Borel,

\item either $n =0$ or $\{\xbf \in \Omega^{[n]} \mid I_\xbf \not \in \Ical\}$
is full, and

\item whenever $\xbf \in \Omega^{[n]}$ and $I_\xbf \not \in \Ical$, $I_\xbf = \Omega$.
\end{itemize}

\end{defn}

The utility in this definition comes from the fact that if $B \in \Hcal_n$ and $E \in \Ical_n$ is full,
then $B \setminus E$ is in $\Hcal_n$.
We also have the following lemma relating $\Ical_n$ and the $\Hcal_n$-meager sets.

\begin{lem} \label{full_cover}
Every element of $\Ical_n$ is contained in a full element of $\Ical_n$.
In particular, every element of $\Ical_n$ is $\Hcal_n$-nowhere dense.
\end{lem}

\begin{proof}
The proof is by induction on $n$.
By convention, $\emptyset$ is a full element of $\Ical_0$.
Suppose now that $Z \in \Ical_{n+1}$ is given.
By replacing $Z$ by a superset, we may assume that $Z$ is Borel.
Let $X$ be a full set in $\Ical_n$ which contains
$\{\xbf \in \Omega^{[n]} \mid Z_\xbf \not \in \Ical\}$.
It follows that $Z \cup (X \times \Omega)$ is a full element of $\Ical_{n+1}$ containing $Z$.
\end{proof}

We are now ready to define the notion of $\Hcal_n$-measurability.

\begin{defn}
  A subset $X$ if $\Omega^{[n]}$ is \emph{$\Hcal_n$-measurable} if there is a Borel set $B$ such that
  $B \symdif X : = (B \setminus X) \cup (X \setminus B)$ is $\Hcal_n$-meager.
\end{defn}

Notice that the $\Hcal_n$-measurable sets form a $\sigma$-algebra which includes Borel sets and the $\Hcal_n$-meager sets.
We will prove below that in the presence of a suitable large cardinal hypothesis,
all $\PCA$-sets are $\Hcal_n$-measurable.
This assertion itself will be needed as a hypothesis in some of our results.

\begin{notn} ($\dagger_n$) denotes the hypothesis that if $m \leq n$, every $\PCA$-subset of $\Omega^{[m]}$ is $\Hcal_m$-measurable.
($\dagger$) denotes the assertion that ($\dagger_n$) holds for all $n$.
\end{notn}

Much of the relevance of $\PCA$-sets and their measurability comes via the following complexity computation.

\begin{lem} \label{category_quant}
If $Z \subseteq \Omega^{[n+1]}$ is Borel, then both 
\[
\{\xbf \in \Omega^{[n]} \mid Z_{\xbf} \in \Ical\}
\]
and its complement are $\PCA$-sets.
\end{lem}

\begin{proof}
By Lemma \ref{dense_open_basis},
$Z_\xbf \in \Ical$ is equivalent to
  \[
\exists G \in {\widehat \Gscr}\ \,\forall x \in \Omega \ \Big((G \in \Gscr) \land \big((z \in W(G)) \rightarrow (z \not \in Z_\xbf)\big)\Big).
  \]
 By Lemma \ref{Gscr_CA}, $\Gscr$ is a $\CA$-set and hence $\{\xbf \in \Omega^{[n]} \mid Z_{\xbf} \in \Ical\}$ is a $\PCA$-set (see discussion in Section \ref{Sect:prelim}).
 
Next observe that since $Z$ is Borel,  each $Z_{\xbf}$ has the Baire property with respect to $\tau$.
Thus $Z_\xbf \not \in \Ical$ is equivalent to
  \[
  \exists p \in \Hbb \ \exists G \in {\widehat \Gscr}\ \, \forall z \in \Omega\ \Big((G \in \Gscr) \land \big((z \in W(G) \cap N_p) \rightarrow (z  \in Z_\xbf)\big)\Big).
  \]
It follows that $\{\xbf \in \Omega^{[n]} \mid Z_{\xbf} \not \in \Ical\}$ is $\PCA$.
 \end{proof}

We will view each collection $\Hcal_n$ as being ordered by containment.
It will be helpful to borrow the following terminology from forcing.

\begin{defn}
  Two elements of $\Hcal_n$ are \emph{compatible} if their intersection contains an element of $\Hcal_n$;
  otherwise they are \emph{incompatible}.
\end{defn}

\begin{lem} \label{Hn_ccc}
For all $n$, every family of pairwise incompatible elements of $\Hcal_n$ is countable.
\end{lem}

\begin{proof}
It suffices to show that if $Z_0$ and $Z_1$ are in $\Hcal_n$ and have the same stem, then
$Z_0 \cap Z_1$ is in $\Hcal_n$ and has the same stem as them both.
This is proved by induction on $n$.
If $n=0$, this is trivial.
Suppose now that it's true for $n$ and $Z_0,Z_1 \in \Hcal_{n+1}$ have the same stem $(s_0,\ldots,s_n)$.
Set
\[
X_i:= \{\xbf \in \Omega^{[n]} \mid (Z_i)_\xbf \ne \emptyset \}
\]
By our inductive assumption, $X_0 \cap X_1$ is in $\Hcal_n$ and has stem $(s_0,\ldots,s_{n-1})$.
Now suppose that $\xbf \in X_0 \cap X_1$.
Observe that $(Z_0 \cap Z_1)_\xbf = (Z_0)_\xbf \cap (Z_1)_\xbf$.
By assumption $(Z_i)_{\xbf} = N_k(y_i) \setminus E_i$ for $y_0,y_1 \in \Omega$ and $E_i \subseteq \Omega$ is $\tau$-meager such that
$s_n = y_0 \restriction k = y_1 \restriction k$.
We are finished with the observation that $N_k(y_0) \cap N_k(y_1) = N_k(y_0 \join y_1)$
and hence $(Z_0 \cap Z_1)_\xbf = N_k(y_0 \join y_1) \setminus (E_0 \cup E_1)$.
\end{proof}

Lemma \ref{Hn_ccc} has the following important consequence.

\begin{lem} \label{Borel_approx}
  If $X \subseteq \Omega^{[n]}$ is $\Hcal_n$-measurable, 
there are Borel sets $A,B \subseteq \Omega^{[n]}$ such that $A \subseteq X \subseteq B$ and
both $X \setminus A$ and $B \setminus X$ are $\Hcal_n$-meager.
\end{lem}

\begin{proof}
  We begin with a pair of claims which will also be needed later.

  \begin{claim} \label{antichain_complement}
    If $\Acal \subseteq \Hcal_n$ is a maximal family of pairwise incompatible sets, then
    $\Omega^{[n]} \setminus \bigcup \Acal$ is Borel and $\Hcal_n$-nowhere dense. 
  \end{claim}

  \begin{proof}
By Lemma \ref{Hn_ccc}, $\Acal$ is countable and hence
$Y := \Omega^{[n]} \setminus \bigcup \Acal$ is a Borel.
To see that $Y$ is $\Hcal_n$-nowhere dense, let $B \in \Hcal_n$ be arbitrary and let
$A \in \Acal$ be such that $A \cap B$ contains some $C \in \Hcal_n$.
Then $C \subseteq A$ and hence is disjoint from $Y$. 
Since $C \subseteq B$ and $B$ was arbitrary, we are done.
  \end{proof}
  
  \begin{claim} \label{Borel_NWD}
    Every $\Hcal_n$-nowhere dense set is contained
    in a Borel $\Hcal_n$-nowhere dense set.
  \end{claim}

  \begin{proof}
    Let $Z \subseteq \Omega^{[n]}$ be $\Hcal_n$-nowhere dense and
    $\Acal \subseteq \Hcal_n$ be a maximal collection with the properties that $\bigcup \Acal \cap Z = \emptyset$ and that $\Acal$ is pairwise incompatible.
To see that $\Acal$ is moreover maximal with respect to being pairwise incompatible, suppose that $B \in \Hcal_n$
and let $C \subseteq B$ be in $\Hcal_n$ and disjoint from $Z$.
Since $C$ is compatible with some element of $\Acal$, so is $B$.
It now follows from Claim \ref{antichain_complement} that
$Y:= \Omega^{[n]} \setminus \bigcup \Acal$ is a Borel $\Hcal_n$-nowhere dense set
which contains $Z$.
\end{proof}

Observe that we have also established that every $\Hcal_n$-meager set is contained in a Borel
$\Hcal_n$-meager set.
Now let $X$ be given as in the statement of the lemma
and let $C \subseteq \Omega^{[n]}$ be a Borel
set such that $X \symdif C$ is $\Hcal_n$-meager.
Let $E \subseteq \Omega^{[n]}$ be a Borel $\Hcal_n$-meager set which contains
$X \symdif C$.
It follows that $A := C \setminus E$ and $B := C \cup E$ satisfy the conclusion of the Lemma.
\end{proof}

When combined with ($\dagger$), Lemma \ref{Borel_approx} also allows us to reduce the complexity of $\PCA$-functions
at the cost of removing an $\Hcal_n$-meager set. 

\begin{lem}[$\dagger_n$] \label{PCA_fn_meas} 
If $f$ is a partial $\PCA$-function from $\Omega^{[n]}$ to $\Omega$,
then there is a Borel $B \subseteq \dom(f)$
such that $f \restriction B$ is Borel and $\dom(f) \setminus B$ is $\Hcal_n$-meager.
\end{lem}

\begin{proof}
  Let $f$ be given as in the statement of the lemma.
  Observe that for each $s \in \Sigma$, 
  $$f^{-1}([s]) = \{\xbf \in \Omega^{[n]} \mid \exists y \in [s]\ ((\xbf,y) \in f) \}$$ is $\PCA$.
  By our hypothesis and Lemma \ref{Borel_approx}, there is a Borel set $B_s \subseteq f^{-1}([s])$ such that
  $f^{-1}([s]) \setminus B_s$ is $\Hcal_n$-meager.
  Define $$B = \bigcap_{n=0}^\infty \bigcup_{|s| = n} B_s$$
  and observe that $B \subseteq \dom(f)$ and
  \[
  \dom(f) \setminus B \subseteq \bigcup_{s \in \Sigma} f^{-1}([s]) \setminus B_s
    \]
  is
  $\Hcal_n$-meager.
  Furthermore, $B \cap f^{-1}([s]) = B \cap B_s$ is a Borel set for each $s$.
  Thus $f \restriction B$ is Borel.
\end{proof}

\begin{lem} \label{Hn_Fubini}
Assume $(\dagger_n)$.
The $\Hcal_{n+1}$-meager sets coincide with $\Ical_{n+1}$.
Moreover:
\begin{enumerate}

\item The $\Hcal_{n+1}$-meager and $\Hcal_{n+1}$-nowhere dense sets coincide.

\item \label{Hn_nonmeager} No element of $\Hcal_{n+1}$ is $\Hcal_{n+1}$-meager.

\item \label{Hn_gen}
If $Z$ is a Borel $\Hcal_{n+1}$-nonmeager, then $Z$ contains an element of $\Hcal_{n+1}$.

\end{enumerate}
\end{lem}

\begin{proof}
The proof is by induction on $n$.
To see the base case $n=0$, observe that the $\Hcal$-meager sets coincide with the collection $\Ical$ which consists
of the $\tau$-meager sets.
Since every $\tau$-meager set is contained in a Borel $\tau$-meager set, every
$\Hcal$-meager set is contained in a Borel $\Hcal$-meager set.
It follows that every $\Hcal$-meager set is $\Hcal$-nowhere dense.
Since $\tau$ is Choquet and hence Baire, no element of $\Hcal$ is $\Hcal$-meager.
Finally, if $Z$ is a Borel $\Hcal$-nonmeager set, then $Z$ has the Baire property with respect to $\tau$ and
therefore contains an element of $\Hcal$.
This establishes the base case of the lemma via our convention of identifying $\Omega$ with $\Omega^{[1]}$.

Now suppose $n > 0$.
By Lemma \ref{full_cover}, we know that every element of $\Ical_{n+1}$ is $\Hcal_{n+1}$-nowhere dense.
We will first prove that if $Z$ is a Borel set not in $\Ical_{n+1}$, then $Z$ is not $\Hcal_{n+1}$-nowhere dense.
Toward this end, suppose that $X:=\{\xbf \in \Omega^{[n]} \mid Z_\xbf \not \in \Ical\}$ is not
in $\Ical_n$.
By our induction hypothesis, $X$ is $\Hcal_n$-nonmeager.
Define $R$ to consist of all $(\xbf,p,G) \in \Omega^{[n]} \times \Hbb \times \widehat{\Gscr}$ such that for all $z \in \Omega$:
\begin{itemize}

\item $G \in \Gscr$ and

\item if $z \in N_p \cap W(G)$, then $(\xbf,z) \in Z$.

\end{itemize}
Observe that for $\xbf \in \Omega^{[n]}$, $\xbf \in X$ if and only if there exist
$p$ and $G$ such that $(\xbf,p,G) \in R$.
$R$ is a $\CA$-relation and therefore by Theorem \ref{KNU}, there is a $\CA$-function
$\psi: X \to \Hbb \times \Gscr$ whose graph is contained in $R$.
Let $s \in \Sigma$ be such that 
\[
X_s:= \{\xbf \in X \mid \exists y \in \Omega\ \exists G \in {\widehat \Gscr}\ \,(\psi(x) = ((s,y),G))\}
\]
is $\Hcal_n$-nonmeager.
Since $X_s$ is $\PCA$, ($\dagger_n$) implies that $X_s$ is $\Hcal_n$-measurable.
By Lemmas \ref{Borel_approx} and \ref{PCA_fn_meas} and our induction hypothesis, $X_s$ contains an $A \in \Hcal_n$ such
that $\psi \restriction A$ is Borel.
Define $B$ to be all $(\xbf,z) \in \Omega^{[n+1]}$ such that $\xbf\in A$ and if $\psi(\xbf) = (p,G)$, then $z \in N_p \cap W(G)$.
To see that $B$ is Borel, observe that is both $\analytic$ and $\CA$:
\begin{itemize}
\item $(\xbf,z) \in B$ if and only if there is a $(p,G) \in \Hbb \times {\widehat \Gscr}$ such that $(\xbf,p,G) \in \psi \restriction A$ and $z \in N_p \cap W(G)$

\item $(\xbf,z) \in B$ if and only if for all $(p,G) \in \Hbb \times {\widehat \Gscr}$, if $(\xbf,p,G) \in \psi \restriction A$, then $z \in N_p \cap W(G)$.

\end{itemize}
It follows that $B \in \Hcal_{n+1}$ and $B \subseteq Z$.
In particular, $Z$ is not $\Hcal_{n+1}$-nowhere dense.

Notice that we have established (by way of contraposition), that if $Z \subseteq \Omega^{[n+1]}$
is Borel and $\Hcal_{n+1}$-nowhere dense, then $Z \in \Ical_{n+1}$.
By Claim \ref{Borel_NWD} and countable additivity of $\Ical_{n+1}$, it follows that
every $\Hcal_{n+1}$-meager set is in $\Ical_{n+1}$.
Since every element of $\Ical_{n+1}$ is $\Hcal_{n+1}$-nowhere dense, this also
establishes the first auxiliary conclusion of the lemma.
Since no element of $\Hcal_{n+1}$ is $\Hcal_{n+1}$-nowhere dense,
it also follows that no element of $\Hcal_{n+1}$ is $\Hcal_{n+1}$-meager.
To see the remaining auxiliary conclusion of the lemma,
suppose that $Z$ is Borel and $\Hcal_{n+1}$-nonmeager.
We have established that $Z \not \in \Ical_{n+1}$ and
therefore that there is a $B \subseteq Z$ in $\Hcal_{n+1}$.
\end{proof}

We will need the following generalization of the classical \emph{Banach-Mazur game} (see \cite{kechris}).
If $\Fcal$ is a collection of nonempty sets ordered by containment and $X \subseteq \bigcup \Fcal$,
then the \emph{Banach-Mazur game} associated
to $(X,\Fcal)$ is defined as follows.
Two players Nonempty and Empty alternately play a $\subseteq$-decreasing sequence of elements of $\Fcal$,
with Nonempty making the first move.
Nonempty wins a play of the game if the intersection of the sequence of plays has nonempty intersection with $X$.

\begin{lem}\label{nonmeager}
A subset $X$ of $\Omega^{[n]}$ is $\Hcal_n$-nonmeager
if and only if Empty does not have a winning strategy in the Banach-Mazur game played on
$(X,\Hcal_n)$.
\end{lem}
\begin{proof}
This follows the standard proof for the case of the ideal of meager subsets of a topological space.
For example, if $X$ is the union of a sequence $X_k$ $(k<\omega)$ of $\Hcal_n$-nowhere dense sets,
Empty's winning strategy at stage $l$ is to play a set in $\Hcal_n$
which is disjoint from $\bigcup_{k \leq l} X_k$.
Conversely, if $\sigma$ is a strategy of Empty, we build a sequence
$\Acal_k$ $(k<\omega)$ of maximal antichains of $\Hcal_n$ so that $\Acal_{k+1}$
refines $\Acal_k$
for all $k$ and so that to every branch $\Bcal$ of the tree
$(\bigcup_{k=0}^\infty \Acal_k, \supseteq)$ there corresponds a play of the Banach-Mazur game on $(X,\Hcal_n)$
in which Empty uses $\sigma$.
This in particular means that we have $(\bigcap \Bcal)\cap X=\emptyset$
for every branch $\Bcal$ of $(\bigcup_{k =0}^\infty \Acal_k, \supseteq)$.
It follows that $X$ is covered by the collection $$\{\Omega^{[n]}\setminus \bigcup \Acal_k \mid k \in \omega\}$$
which, by Claim \ref{antichain_complement}, consists of $\Hcal_n$-nowhere dense sets.
\end{proof}

\begin{prop}[$\dagger_{n-1}$]\label{ub}
Every universally Baire subset of $\Omega^{[n]}$
is $\Hcal_n$-measurable.
\end{prop}

\begin{proof}
Fix a universally Baire subset $Z$ of $\Omega^{[n]}$.
Let $\Bbb$ be the Boolean algebra of all Borel subsets of $\Omega^{[n]}$
modulo the $\Hcal_n$-meager sets.
For any Borel set $B \subseteq \Omega^{[n]}$, write $[B]$ for the element of $\Bbb$ it represents.
Since $\Bbb$ is c.c.c.\ by Lemmas \ref{Hn_ccc} and \ref{Hn_Fubini} and since it is countably complete, it is a complete Boolean algebra.
By Lemmas \ref{Hn_ccc} and \ref{Hn_Fubini},
every positive element of $\Bbb$ is of the form $[B]$ for some $B$ which is a countable union
of sets in $\Hcal_n$.

Let $K$ be the Stone space consisting of all ultrafilters on $\Bbb$.
Define
$X \subseteq K$ to consist of all $p \in K$ such that for some $f(p) \in \Omega^{[n]}$,
$[U] \in p$ whenever $U \subseteq \Omega^{[n]}$ is a metrically open set containing $f(p)$.
Equivalently, $p$ is in $X$ if and only if for every $\epsilon >0$, there is a Borel $B \subseteq \Omega^{[n]}$ with diameter less than $\epsilon$
such that $[B] \in p$ (here we have fixed a complete, compatible metric on $\Omega^{[n]}$).
Notice that, for a given $p \in X$, $f(p)$ is necessarily unique and that
the function $f:X \to \Omega^{[n]}$ is continuous.
Observe, for any $\epsilon > 0$, $\Omega^{[n]}$ can be covered by countably many Borel sets of diameter at most $\epsilon$.
Since the ideal of $\Hcal_n$-nonmeager sets is countably additive, Lemma \ref{Hn_Fubini} implies that
every Borel subset $B$ of $\Omega^{[n]}$ which is not $\Hcal_n$-meager 
contains an element of $\Hcal_n$ of arbitrarily small diameter (which is itself not $\Hcal_n$-meager by (\ref{Hn_nonmeager})
of Lemma \ref{Hn_Fubini}).
Thus any positive element of $\Bbb$ is contained in
a filter $\Fcal$ such that for every $\epsilon > 0$, there is a $B \in \Hcal_n$ of diameter less than $\epsilon$
such that $[B] \in \Fcal$.
Since any ultrafilter extending $\Fcal$ is in $X$, it follows that $X$ is dense.
Since $Z$ is universally Baire, the preimage $Y:=f^{-1}(Z)$ has the property of Baire in $X$.
Fix a regular open subset $V$ of $K$ such that $M:=Y \symdif (V\cap X)$
is meager in $X$.
Since $\Bbb$ is complete, $V$ corresponds to an element of $\Bbb$.
Let $\Acal \subseteq \Hcal_n$ be a countable pairwise disjoint family
such that $V = [\bigcup \Acal]$.

It suffices to show that
$Z$ differs from the Borel set $\bigcup \Acal$ by an $\Hcal_n$-meager set.
Suppose for contradiction that one of the sets $\bigcup \Acal \setminus Z$ or
$Z\setminus \bigcup \Acal$ is not $\Hcal_n$-meager.
If  $\bigcup \Acal \setminus Z$ is not $\Hcal_n$-meager,
then for some $A\in \Acal$ the set $A\setminus Z$ is not $\Hcal_n$-meager.
Then using Lemma \ref{Hn_Fubini}, we can find a run of the Banach-Mazur game on $\Hcal_n \restriction (A\setminus Z)$
with intersection $x$
such that $f^{-1}(x)$ is disjoint from $M$.
This would show that $V\cap X$ has a point belonging to neither the set $Y$ nor $M$, a contradiction.
On the other hand, if the set $Z\setminus \bigcup \Acal$ is not $\Hcal_n$-meager, we could find $B\in \Hcal_n$ incompatible with every member of $\Acal$ such that $B\cap Z$ is not in $\Hcal_n$-meager.
Applying the Banach-Mazur argument again, we find a point $x$ in $B\cap Z$ whose preimage $f^{-1}(x)$ is disjoint from $M$.
This yields a point of $Y$ that does not belong to either of the sets $V$ or $M$, a contradiction.
\end{proof}

\begin{prop} \label{dagger_prop}
  If for every $a \subseteq \omega$, $a^\sharp$ exists, then {\emph{($\dagger$)}} holds.
  In particular, \emph{($\dagger$)} follows from the existence of a measurable cardinal.
\end{prop}

\begin{rem}
Brendle and L\"owe \cite{BL}
have shown that ($\dagger_1$) is equivalent to the assertion that $\aleph_1$ is an inaccessible cardinal
in $L[r]$ for any $r \subseteq \omega$.
\end{rem}

\begin{proof}
  We will verify ($\dagger_n$) inductively.
  Notice that ($\dagger_0$) is a vacuous statement.
  Given ($\dagger_n$), Theorem \ref{ub} implies that all universally Baire subsets of $\Omega^{[n+1]}$ are $\Hcal_{n+1}$-measurable.
  By \cite[Theorem 3.4]{UB}, every $\PCA$-set subset of $\Omega^{[n+1]}$ is universally Baire and therefore ($\dagger_{n+1}$) holds.
\end{proof}

\begin{cor}
  If there is a supercompact cardinal or a proper class of Woodin cardinals, then every subset of
  $\Omega^{[n]}$ belonging to the inner model $L(\Rbb)$ is $\Hcal_n$-measurable.
\end{cor}

\begin{proof}
In \cite{W88} (see also \cite{UB}), 
Woodin showed that either large cardinal hypothesis implies 
every subset of $\Omega^{[n]}$ belonging to the inner model  $L(\Rbb)$ is universally Baire.
Since either large cardinal hypothesis implies that for every $x \subseteq \omega$, $x^\sharp$ exists,
the corollary therefore follows from Propositions \ref{ub} and \ref{dagger_prop}.
\end{proof}

Thus, assuming a standard large cardinal axiom, the inner model $L(\Rbb)$ is a natural model of the statement that all subsets of $\Omega^{[n]}$ are $\Hcal_n$-measurable.
However, as in the case of Lebesgue measurability, if we are interested just in the consistency of this statement, we can reduce the large cardinal assumption considerably (still, the large cardinal assumption is more substantial
than in the case of Lebesgue measurability).

\begin{prop} \label{Hn_coll}
If $\mu$ is a measurable cardinal and $\kappa < \mu$ is inaccessible, then in every generic extension
  by the Levy collapse $\Coll (\omega, \kappa)$, all subsets of $\Omega^{[n]}$ definable from an $\omega$-sequence of ordinals are $\Hcal_n$-measurable.
\end{prop}

\begin{proof}
Let $\ON$ denote the class of ordinals.
While we will need to utilize the measurable cardinal to ensure that $\Hcal_n$ and $\Ical_n$ are sufficiently absolute and that ($\dagger$) holds in various generic
extensions,
we will otherwise closely follow the modern expositions of Solovay's original proof \cite{So70} for Lebesgue measurability of
such sets in the Levy collapse forcing extension (see, for example, \cite{higher_infinite}).
Let $G$ be the generic filter of $\Coll (\omega, \kappa)$.
Note that by \cite{LevySolovay}, any intermediate generic extension of $V[G]$ by $\Coll(\omega,\delta)$ for $\delta < \kappa$
satisfies that $\mu$ is a measurable cardinal and therefore
that ($\dagger$) holds by Proposition \ref{dagger_prop}.

If $B \subseteq \Omega^{[n]}$ is a Borel set, then both $B$ and its complement are projections of trees on $\omega^n \times \omega$.
We say that this pair of trees is the \emph{code} of $B$.
We will always view Borel sets in a given model as being constructed using their code.
So, for instance, $V[G] \models B \in \Hcal_n$ is the assertion that the Borel set described by $B$'s code is in $V[G]$'s interpretation
of the definable set $\Hcal_n$.
Notice that while codes for a given Borel set are not unique, the interpretation does not depend on the choice of code.

\begin{claim} \label{HnIn_abs}
  Suppose that $a \in \ON^\omega \cap V[G]$ and $B \subseteq \Omega^{[n]}$ is a Borel set coded in $V[a]$.
  The following are true:
  \begin{enumerate}
    
  \item \label{In_abs} $V[a] \models B \in \Ical_n$ if and only if $V[G] \models B \in \Ical_n$;
    
  \item \label{Hn_abs} if $V[a] \models B \in \Hcal_n$, then $V[G] \models B \in \Hcal_n$.
  \end{enumerate}
\end{claim}

\begin{proof}
Observe that for a given $B$, since $V[a] \models (\dagger)$, the reverse implication in (\ref{In_abs}) follows from the forward implication in (\ref{In_abs}) together with (\ref{Hn_abs}):
if $V[a] \models B \not \in \Ical_n$, then by Lemma \ref{Hn_Fubini}, $V[a] \models \exists A \in \Hcal_n \ (A \subseteq B)$ which by (\ref{Hn_abs}) implies $V[G] \models \exists A \in \Hcal_n \ (A \subseteq B)$
 and thus by Lemma \ref{Hn_Fubini}, $V[G] \models A \not \in \Ical_n$ and hence $V[G] \models B \not \in \Ical_n$. 
  With this in mind, we will prove the claim by induction on $n$.
  Notice that the case $n=0$ is vacuously true.
  Now suppose that the claim holds for $n$ and that $B \subseteq \Omega^{[n+1]}$ is coded in $V[a]$.
  If $V[a] \models B \in \Ical_{n+1}$,
  then there is a Borel set $A$ coded in $V[a]$ such that
  $$
  V[a] \models (A \in \Ical_n) \land \forall \xbf \in \Omega^{[n]}\ [(\xbf \not \in A) \rightarrow (B_\xbf \in \Ical)]
  $$
  By our induction hypothesis, $V[G] \models A \in \Ical_n$.
  Also $$\forall \xbf \in \Omega^{[n]}\ [(\xbf \not \in A) \rightarrow (B_\xbf \in \Ical)]$$ is a $\mathbf{\Pi}^1_2$-sentence
  with parameters in $V[a]$.
  It follows from Shoenfield's absoluteness theorem \cite{Sho} (see also \cite[13.15]{higher_infinite}) that
  $$V[G] \models \forall \xbf \in \Omega^{[n]}\ [(\xbf \not \in A) \rightarrow (B_\xbf \in \Ical)]$$
  and hence that $V[G] \models B \in \Ical_{n+1}$.

  Next suppose that $$V[a] \models B \in \Hcal_{n+1} \textrm{ with stem }(s_0,\ldots,s_n).$$
  Let $A$ be the Borel set with code in $V[a]$ such that
  $$V[a] \models (A = \{\xbf \in \Omega^{[n]} \mid B_\xbf \ne \emptyset\}) \land (A \in \Hcal_n).$$
  By our induction hypothesis $V[G] \models A \in \Hcal_n$.
 Observe that $B_\xbf \in \Hcal$ is equivalent to 
 $$\exists y \ \exists k \ \exists G \ \forall z\ [((z \in W(G)) \rightarrow ((z \in N_k(y) \leftrightarrow (z \in B_\xbf))]$$
 (with quantifier ranges as in the proof of Lemma \ref{category_quant}).  
 Thus
  $$\forall \xbf \in \Omega^{[n]} \ [(\xbf \not \in A) \rightarrow (B_\xbf \in \Hcal \textrm{ with stem }s_n)]$$
  is a $\mathbf{\Pi}^1_3$-sentence.
Since $\Coll(\omega,\kappa)$ has cardinality less than a measurable cardinal,
  the Martin-Solovay Absoluteness Theorem \cite{MartinSolovay} (see also \cite[15.6]{higher_infinite}) implies this sentence is satisfied by $V[G]$ and therefore
  $V[G] \models B \in \Hcal_{n+1}$.
\end{proof}

Let $\Qcal_n$ be the poset consisting of all Borel subsets of $\Omega^{[n]}$ which are not in $\Ical_n$.
Observe if $a \in \ON^\omega \cap V[G]$, then since $V[a] \models (\dagger)$, Theorem \ref{Hn_Fubini} yields that
$V[a]$ satisfies $\Hcal_n \subseteq \Qcal_n$ is a dense suborder.
By Claim \ref{HnIn_abs}, $\Qcal_n^{V[a]} = \Qcal_n^{V[G]} \cap V[a]$.

Now let $X$ in $V[G]$ be a subset of $\Omega^{[n]}$ definable from $a \in \ON^\omega$.
Mimicking Solovay's argument, find a formula $\varphi(a,x)$ such that $V[G] \models x \in X$
if and only if $V[a,x] \models \varphi(a,x)$ (see, e.g., \cite[11.12]{higher_infinite}).
Since $\Ical_n$ a $\sigma$-ideal generated by Borel sets and since $\Ical_n \cap V[a]$ is countable in $V[G]$, $E:=\bigcup (\Ical_n \cap V[a])$
is a Borel set in $\Ical_n^{V[G]}$.
From this point forward, $\Qcal_n$ will always be interpreted in $V[a]$.

\begin{claim} \label{Hx_gen}
$V[G]$ satisfies that for any $x \in \Omega^{[n]} \setminus E$,
$H_x := \{B \in \Qcal_n \mid x \in B\} \subseteq \Qcal_n$ is a $V[a]$-generic filter.
\end{claim}

\begin{proof}
To see that $H_x$ is a filter, suppose that $A,B \in H_x$.
Since $x \in A \cap B$, $A \cap B \not \in \Ical_n \cap V[a]$ and therefore $A \cap B \in \Qcal_n$ is a lower bound for $A,B$.
Next suppose that $\Acal \subseteq \Qcal_n$ is a maximal antichain in $V[a]$.
By Claim \ref{antichain_complement}, $E_0:=\Omega^{[n]} \setminus \bigcup \Acal$ is a Borel $\Hcal_n$-nowhere dense set.
By Lemma \ref{Hn_Fubini}, $E_0 \in \Ical_n$ and since $E_0 \in V[a]$, $E_0 \subseteq E$.
Since $x \not \in E_0$, it follows that $x \in A$ for some $A \in \Acal$.
Thus $A \in H_x \cap \Acal$ and we have shown that $H_x$ is $V[a]$-generic.
\end{proof}

\begin{claim}
Every condition of $\Qcal_n$ forces that the intersection of the generic filter is a singleton.
\end{claim}

\begin{proof}
For each $k$, define $\Dcal_k \subseteq \Qcal_n$ to consist of all $B$ such that for all $\xbf,\ybf \in B$ and $i < n$,
$x_i \restriction k = y_i \restriction k$.
Any element $B$ of $\Qcal_n$ is a countable union of sets in $\Dcal_k \cup \Ical_n$.
Since $\Ical_n$ is closed under taking countable unions, at least one of these sets must be in $\Dcal_k$.
Thus $B$ has a subset in $\Dcal_k$ and since $B$ was arbitary, it follows that $\Dcal_k$ is dense.
It follows that every condition of $\Qcal_n$ forces that the intersection of the generic filter has at most one element.

If $B \in \Qcal_n$ is any condition, let $x \in B \setminus E$.
Then $H_x \subseteq \Qcal_n$ is $V[a]$-generic filter containing $B$ such that $\bigcap H_x = \{x\} \ne \emptyset$.
Thus $B$ can not force that the intersection of the generic filter is empty.
Since no condition forces that the intersection of the generic filter is empty, every condition must force that it is nonempty.
\end{proof}

Now let $\Acal \subseteq \Qcal_n \cap V[a]$ be a maximal antichain consisting of $A$ such that $A$ decides $\varphi(\check a,\dot x)$ where
$\dot x$ is the name for the unique element of the intersection of the $V[a]$-generic filter for $\Qcal_n$. 
Let $Y$ be the union of those elements of $\Acal$ which force $\varphi(\check a,\dot x)$.
Since $\Acal$ is countable, $Y$ is Borel.
It therefore suffices to show that $X \symdif Y \subseteq E$.
To see this, suppose that $x \in \Omega^{[n]} \setminus E$.
By Claim \ref{Hx_gen}, $H_x$ is $V[a]$-generic and in particular, there is an $A \in \Acal$ with $x \in A$.
It follows that $x \in X$ if and only if $V[a][x] \models \varphi(a,x)$ if and only if $A$ forces $\varphi(\check a, \dot x)$ if and only if $x \in A \subseteq Y$.
\end{proof}

\section[Measurable Partition Hypothesis]{The Partition Hypothesis for measurable partitions}
\label{Sect:MeasPartHyp}

Our goal in this section is to prove the Partition Hypothesis holds for partitions which are $\Hcal_n$-measurable.
In fact, the $n$-cofinal function which witnesses
the conclusion of the Partition Hypothesis can be taken to be $\PCA$.

\begin{thm}[$\dagger$] \label{partition}
Suppose $n \geq 0$ and $X \subseteq \Omega^{[n]}$ is $\Hcal_n$-measurable and not $\Hcal_n$-meager.
If $c: X \to \omega$ is $\Hcal_n$-measurable, then there is an $n$-cofinal function $F:\Omega^{\leq n} \to \Omega$ which is $\PCA$ such that the
range of $F^*$ is contained in $X$ and $c \circ F^*$ is constant.
\end{thm}

\begin{proof}
The proof is by induction on $n$.
The bulk of the work will be in establishing the
following claim, whose proof will utilize the induction hypothesis of the theorem.

\begin{claim}[$\dagger$] \label{nonmeager_to_n-cofinal}
Suppose $n \geq 0$.
If $Z \subseteq \Omega^{[n]}$ is Borel and
$\Hcal_n$-nonmeager, then there is an
$n$-cofinal function $F$ which is $\PCA$ such that the range of $F^*$ is contained in $Z$.
\end{claim}

\begin{proof}
If $n= 0$, we note that $\Omega^{[n]}$ consists only of the null sequence and
$\Omega^{[\leq n]} := \bigcup_{k=1}^n \Omega^{[k]}$ is empty.
Thus the base case of the theorem is vacuously true.

Next suppose that $n=1$ and $Z \subseteq \Omega = \Omega^{[1]}$ is Borel and nonmeager.
Since $Z$ is Borel, it has the Baire property with respect to $\tau$ and there exist $G \in \Gscr$ and $p \in \Hbb$ such
that $N_p \cap W(G) \subseteq Z$.
Set $k = |s_p|$ and
for each $z \in \Omega$, let $y= y(z) \in \Omega$ be defined by
$$
y(i) := 
\begin{cases}
x_p (i) & \textrm{ if } i < k \\
\max(x_p(i),z(i)) & \textrm{ if } i \geq k \\
\end{cases}.
$$
Observe that $y:\Omega \to N_p$ is continuous and $y \restriction N_p$ is the identity.
Set $F(z) := g(G,(s_p, y(z)))$.
Since $g$ is $\CA$, $F$ is $\PCA$.
Moreover, it is vacuously true that $F$ is increasing with respect to $\trleq$.
Since $g(G,(s_p,y(z)))$ is in $N_k(y(z))$, we have $z \leq^k y(z) \leq^0 F(z)$ and hence
$z \leq^k F(z)$.
Finally, $\Omega^{[\leq 1]} = (\Omega^{[1]})^1$ which we have identified with $\Omega$.
If $z \in \Omega$, then
$F^*(z)= F(z) = g(G,(s_p,y(z)))$ is in $N_k(y(z)) \cap W(G) \subseteq Z$.

Now suppose that the theorem is true for a given $n \geq 1$.
Let $Z \subseteq \Omega^{[n+1]}$ be Borel and $\Hcal_{n+1}$-nonmeager.
By Lemma \ref{Hn_Fubini}, 
\[
\{\xbf \in \Omega^{[n]} \mid Z_\xbf \textrm{ is }\Hcal\textrm{-nonmeager}\}
\]
contains a Borel set $X_0$ which is $\Hcal_n$-nonmeager.
Define
$R \subseteq X_0 \times \Hbb \times \Gscr$
to consist of all $(\xbf,p,G)$ such that $\xbf \in X_0$
and for all $y \in \Omega$, $y \in N_p \cap W(G)$ implies
$(\xbf, y) \in Z$.
Since $X_0$ and $Z$ are Borel, $R$ is $\CA$.
Observe that since Borel sets are Baire measurable with respect to $\tau$,
if $\xbf \in X_0$, then there exist $p$ and $G$ such that $(\xbf,p,G) \in R$.

By Theorem \ref{KNU}, there are functions 
$\bar p:X_0 \to \Hbb$ and
$\bar{G}:X_0 \to \Gscr$ such that $\xbf \mapsto (\bar p(\xbf), \bar {G}(\xbf))$ is $\CA$ and
for all $\xbf \in X_0$, 
\[
(\xbf,\bar p(\xbf),\bar G(\xbf)) \in R.
\]
Observe that this implies $\bar p$ and $\bar G$ are all $\PCA$-functions.
By Lemma \ref{PCA_fn_meas}, we may find a Borel set $X \subseteq X_0$ which is $\Hcal_n$-nonmeager such that
$\bar p \restriction X$ and $\bar G \restriction X$ are Borel functions.
By the inductive hypothesis for the theorem,
there is a $\PCA$-function $\bar F:\Omega^{\leq n} \to \Omega$ and
$k \in \omega$ and $s \in \Sigma$ such that:
\begin{itemize}

\item $\bar F$ is $n$-cofinal with respect to $\leq^k$;

\item the range of $\bar F^*$ is contained in $X$;

\item the first coordinate of $\bar p \circ \bar F^*$ is constant, taking
the value $s$.

\end{itemize}
For $\xbf \in \Omega^{n+1}$, let $\sigma^0(\xbf), \ldots, \sigma^m (\xbf)$ list the elements of
\[
\{ \sigma \in  \Omega\bb{n} \mid \sigma(n-1) \trleq \xbf\} = \{\sigma \in \Omega\bb{n} \mid (\sigma,\xbf) \in \Omega\bb{n+1}\}.
\]
Here we choose the order of the $\sigma^j(\xbf)$'s to be increasing with respect to some suitable
lexicographic order so that the maps $\xbf \mapsto \sigma^j(\xbf)$ are each continuous.
Define $F:\Omega^{\leq n+1} \to \Omega$ by setting
$F \restriction \Omega^{\leq n} := \bar F$ and
\[
F(\xbf) := \tilde g \left( \Big(\bar G\big(\bar F^* (\sigma^0(\xbf))\big),\ldots,\bar G\big(\bar F^*(\sigma^m(\xbf))\big)\Big),
\Big(s,\max_j x_{\bar p (\bar F^* (\sigma^j (\xbf))} \Big)\right).
\]
Observe that since the first coordinate of $\bar p \circ \bar F^*$ is constantly $s$, 
$s$ is an initial part of $x_{\bar p(\bar F^* (\sigma^i (\xbf)))}$ for each $i$ and
$U:=\bigcap_{i=0}^m N_{\bar p (\bar F^* (\sigma^i (\xbf)))}$ is a nonempty $\tau$-open set
which contains $F(\xbf)$.

We now wish to show that $F$ satisfies the conclusion of the theorem.
Since $F$ is obtained by composing $\PCA$-functions, it is $\PCA$.
Since $n \geq 1$, $F \restriction \Omega = \bar F \restriction \Omega$ and therefore
$x \leq^k F(x)$ whenever $x \in \Omega$.

In order to see $F$ is increasing, it suffices to show that if $\xbf \in \Omega^{n+1}$ and
$0 \leq i \leq n$, then $F(\xbf^i) = \bar F(\xbf^i) \leq F(\xbf)$ where $\xbf^i$ is the result of removing the $i$th least element of $\xbf$.
Fix $i$ and let $j \leq m$ be such that $\sigma^j(\xbf)$ contains $\xbf^i$.
Observe that the maximum (and final) coordinate of $\bar F^*(\sigma^j(\xbf))$ is $\bar F(\xbf^i)$.
By definition, $F(\xbf)$ is an element of 
$N_{\bar p (\bar F^*(\sigma^j(\xbf)))}$.
For each $\ybf \in \Omega^{n}$, since $x_{\bar p(\ybf)}$ is in $W(\bar G) \cap N_{\bar p(\ybf)}$, it is in $Z_{\ybf}$
and in particular dominates each coordinate of $\ybf$.
It follows that $F(\xbf^i) \leq F(\xbf)$ and that $F$ is increasing.
 
Finally we will show that if $\tau$ is in $\Omega\bb{n+1}$, then $F^*(\tau)$ is in $Z$.
Let $\xbf$ be the final coordinate of $\tau$ and let $j \leq m$ be such that
$\tau = (\sigma^j(\xbf) , \xbf)$.
By hypothesis, $\bar F^*(\sigma^j(\xbf))$ is in $X$.
By definition, $F(\xbf)$ is an element of
$$N_{\bar p (\bar F^*(\sigma^j(\xbf)))} \cap W\Big(\bar G\big(\bar F^*(\sigma^j(\xbf))\big)\Big).$$
Recall that $W\Big(\bar G\big(\bar F^*(\sigma^j(\xbf))\big)\Big)$ was selected so that any element $y$ of this intersection satisfied
that $(\bar F^*(\sigma^j(\xbf)) , y)$ is in $Z$.
Thus letting $y=F(\xbf)$, we have $F^*(\tau) = (\bar F^*(\sigma^j(\xbf)) , y)$ is in $Z$ and we are done.
\end{proof}

To complete the proof of Theorem \ref{partition},
suppose that $c:X \to \omega$ is given as in its
statement.
Since $\Hcal_n$-meager sets are closed under taking
countable unions, there is a $k$ such
that $c^{-1}(k)$ is not $\Hcal_n$-meager.
By Lemma \ref{Borel_approx},
$c^{-1}(k)$ contains a Borel set $X$ which is 
$\Hcal_n$-nonmeager.
By Claim \ref{nonmeager_to_n-cofinal},
there is an $F$ satisfying the conclusion of the theorem.
\end{proof}

As noted, a corollary of Theorem \ref{partition} answers a question appearing in both \cite{B17} and \cite{SVHDLwLC}.
We pause to review this question's main notions.
If $x \in \Omega$, define $I(x):=\{(i,j)\mid j\leq x(i)\}$.
For any $\xbf\in\Omega^{n}$ let $\bigwedge \xbf$ denote the meet of the coordinates of $\xbf$ and set $I(\xbf):= I(\bigwedge \xbf)$.
The following definition appears in \cite{B17}; all sums and comparisons therein are taken over the intersections of the relevant functions' domains.
\begin{defn}
\label{def:classical_n_coherence}
Fix $n>0$. A collection $\Phi=\{\varphi_{\xbf}  \mid\xbf\in\Omega^{n}\}$ with each $\varphi_\xbf \in \Zbb^{I(\xbf)}$ is \emph{n-coherent} if
$$\sum_{i=0}^{n}(-1)^i \varphi_{\xbf^i} \restriction I(\xbf)=^* 0$$
for all $\xbf\in\Omega^{n+1}$, where $=^*$ denotes mod finite equality.
$\Phi$ is \emph{trivial} if
\begin{itemize}
\item $n=1$ and there exists a $\psi:\omega^2\to\Zbb$ such that
$$\varphi_x=^* \psi \restriction I(x)$$
for all $x\in\Omega$, or
\item $n>1$ and there exists a $\Psi=\{\psi_{\xbf}\mid\xbf\in\Omega^{n-1}\}$ such that
$$\varphi_{\xbf} =^* \sum_{i=0}^{n-1}(-1)^i\psi_{\xbf^i} \restriction I(\xbf) $$
for all $\xbf\in\Omega^{n}$.
\end{itemize}
\end{defn}
It will sometimes be useful to view an element of $\Zbb^{I(\xbf)}$ as an element of $P:=\Zbb^{\omega \times \omega}$ via the convention that a function
takes the value $0$ outside of its domain.
Thus any collection $\Phi$ as in Definition \ref{def:classical_n_coherence} is naturally viewed as a subset of the product of $\Omega^{n}$ with the Polish space
$P$, and it was in exactly this sense that Todorcevic showed that any analytic $1$-coherent family $\Phi$ is trivial (\cite{To98}). The question appearing in both \cite{B17} and \cite{SVHDLwLC} was whether an analytic $n$-coherent family could be nontrivial for any $n>1$.
The next corollary provides a strong answer to this question under the hypothesis ($\dagger$).
\begin{cor}[$\dagger$]
  \label{cor:CA_cohere->PCA_triv}
Every analytic $n$-coherent family of functions $\Phi$ admits a $\PCA$ trivialization.
\end{cor}

\begin{proof} 
For any partial function $\varphi$ from $\omega^2$ to $\Zbb$, write $s[\varphi]$ for the restriction of $\varphi$ to its support.
The map $\varphi \mapsto s[\varphi]$ is Borel.
Define $c$ on $\Omega^{n+1}$ by
$$
c(\xbf) := s[d\Phi (\xbf)] = s\left[ \sum_{i=0}^n (-1)^i \varphi_{\xbf^i} \restriction I(\xbf) \right].
$$
Since $\Phi$ is analytic, so is $c$.
Since $\Phi$ is $n$-coherent, the range
of $c$ is contained in the collection of finite partial functions
from $\omega^2$ into $\Zbb$ and in particular
is countable.
By Theorem \ref{partition}, there exists a $n$-cofinal $F:\Omega^{\leq n+1}\to\Omega$ which is $\PCA$
such that $c\circ F^*$ is constant.
The existence of such an $F$ is sufficient for standard trivialization constructions which we'll review in a more general setting in Section \ref{Sect:PartHypAdd} below.
More precisely, $F$ will witness exactly that instance of $\PH_n$ which is applied in the trivialization argument of Theorem \ref{coherent_trivial} below, within which the triviality of families like $\Phi$ figure as a special case.
Since the trivialization constructed therein is composed of $\Phi$, the function $F$, sums, and the operations $*$ and $d$ (also defined in the proof of Theorem \ref{coherent_trivial}), the conclusion of 
the corollary follows from the closure
of the $\PCA$-functions under composition.
\end{proof}

A modification of this proof also yields
the following corollary.
\begin{cor}[$\dagger$] \label{cor:UB_cohere->UB_triv}
Every universally Baire $n$-coherent family of functions admits a universally Baire trivialization.
\end{cor}

\begin{proof}
The proof is the same are for Corollary
\ref{cor:CA_cohere->PCA_triv} except for a few minor
modifications:
\begin{itemize}

\item The function $c$ is a composition of universally Baire functions and hence is universally Baire under the hypothesis.

\item By Proposition \ref{ub}, univerally Baire subsets of
$\Omega^{[n]}$ are $\Hcal_n$-measurable.

\item Under the hypothesis, $\PCA$-functions are
universally Baire \cite{UB}.

\end{itemize}
\end{proof}

Corollary \ref{cor:UB_cohere->UB_triv} admits the following interpretation: the families $\Phi$ and $\Psi$ of Definition \ref{def:classical_n_coherence} correspond to cocycles and coboundaries of the standard cochain complex $\mathcal{K}(\Bbf/\Abf)$ for computing the derived limits of a well-studied inverse system termed $\Abf$ in the literature (see \cite{MP,B17,SVHDL}).
Combining Corollary \ref{cor:UB_cohere->UB_triv} with Theorem \ref{rel_UB}, we have the following result.
\begin{cor}
Suppose there is
a supercompact cardinal or a proper class of Woodin cardinals.
The model $L(\Rbb)$ satisfies that ${\lim}^n \Abf = 0$ for all $n$.
\end{cor}

As the proof of Corollary \ref{cor:CA_cohere->PCA_triv} suggests, we will soon follow \cite{strong_hom_add} in adopting more general notions of $n$-coherence than those of Definition \ref{def:classical_n_coherence}; as these notions indeed subsume the classical ones, there is no danger of terminological confusion.
These notions apply to a broad class of inverse systems indexed by $\Omega$, and although it seems reasonable to expect a statement like that of Corollary \ref{cor:UB_cohere->UB_triv} to hold for this more general class of systems, the tasks both of framing and arguing such a statement appear to be nontrivial ones.
They likely will comprise the main work of answering Question \ref{ques:def_lim} of our conclusion below. 

\section{Forcing the Partition Hypothesis}
\label{Sect:forcing}
The arguments of \cite{SVHDL} and \cite{strong_hom_add} both invoke strong combinatorial properties of the \emph{weakly compact Hechler model} of Theorem \ref{PHfromHechlers} below.
In this and the following section, we show that these properties may be more simply regarded as the assertion that $\PH_n$ holds for every $n\in\omega$.
Since we will utilize and adapt the results and proofs of \cite{SVHDL} and \cite{strong_hom_add}, we will follow the style and notational conventions of those papers even when they differ
from the present article.

\begin{thm}
\label{PHfromHechlers}
If $\kappa$ is a weakly compact cardinal and $\Hbb_\kappa$ is the length-$\kappa$ finite support iteration of
Hechler forcing, then any generic extension by $\Hbb_\kappa$ satisfies $\PH_n$.
\end{thm}

As noted in Section \ref{Sect:PH}, $\PH_0$ is a ZFC theorem. 
To argue the $n>0$ instances of Theorem \ref{PHfromHechlers} it will suffice to define ``partial witnesses'' $F$ to $\PH_n$ in the sense of Lemma \ref{extensionlemma}; this will be our approach.
Underlying these functions' domains will be  $\leq^*$-cofinal subcollections $\Upsilon$ of the Hechler reals added by the iteration $\Hbb_\kappa$.
Letting $\Hbb_\alpha$ more generally denote the length-$\alpha$ finite support iteration of Hechler posets, the elements of $\Hbb_\kappa$ are finite partial functions $p$ from $\kappa$ for which $\alpha\in\dom(p)$ implies that $p(\alpha)$ is a nice $\Hbb_\alpha$-name for a Hechler condition.
For each $\alpha<\kappa$ let $\dot{z}_\alpha$ be an $\Hbb_\kappa$-name for the $\alpha\Th$ Hechler real added by $\Hbb_\kappa$.

We will require several lemmas from \cite{SVHDL}; the following, which appears also as Lemma 6 in \cite{strong_hom_add}, combines its Lemmas 3.4 and 4.3.
Recall that if $A$ is a set of ordinals, then $[A]^n$ denotes the collection of all $n$-element subsets of $A$ and $[A]^{<n}$ denotes the collection of subsets of $A$
of cardinality less than $n$.
Elements of these collections are identified with their increasing enumerations;
if $\vec{\alpha}$ is a finite set of ordinals, we will write $\alpha_i$ for the $i\Th$-least element of $\vec{\alpha}$.
If $\vec{\alpha}$ is an initial part of $\vec{\beta}$ we will write $\vec{\alpha} \sqsubseteq \vec{\beta}$.
We denote by $D$ the dense subset of $\Hbb_\kappa$ consisting of conditions $q$ such that for all $\eta\in\mathrm{dom}(q)$ the stem of $q(\eta)$ is determined by $q\restriction\eta$; this is the stem which we reference as $s_{q_{\vec{\alpha}}(\eta)}$ in item (1) below.
\begin{lem}
\label{HechlerDeltaSystem}
Let $\kappa$ be a weakly compact cardinal,
$n$ be a positive integer, and let $\langle q_{\vec{\alpha}}\,|\,\vec{\alpha} \in [\kappa]^n\rangle$ be a family of
  conditions in $D\subseteq\Hbb_\kappa$.
  Let $u_{\vec{\alpha}} = \dom(q_{\vec{\alpha}})$.
  Then there is an unbounded set $A \subseteq \kappa$, a family
  $\langle u_{\vec{\alpha}} \mid \vec{\alpha} \in [A]^{<n}
  \rangle$, a natural number $\ell$, and a set of stems $\langle s_i \mid i < \ell \rangle$
  such that:
  \begin{enumerate}
    \item $|u_{\vec{\alpha}}| =
    \ell$ for all $\vec{\alpha} \in [A]^n$, and if $\eta$ is the $i\Th$ element of
    $u_{\vec{\alpha}}$ then $s_{q_{\vec{\alpha}}(\eta)} = s_i$.
    \item $A$ and $\langle u_{\vec{\alpha}} \mid \vec{\alpha} \in [A]^{\leq n} \rangle$
    satisfy 
    \begin{enumerate}
    \item  for all $\vec{\alpha} \in [A]^{<n}$,
    \begin{enumerate}
      \item if $\beta \in A$ and $\vec{\alpha} < \beta$, then
      $u_{\vec{\alpha}} < \beta$,
      \item if $\vec{\beta} \in [A]^{\leq n}$ satisfies
      $\vec{\alpha} \sqsubseteq \vec{\beta}$, then $u_{\vec{\alpha}} \sqsubseteq
      u_{\vec{\beta}}$,
      \item the set $\{u_{\vec{\alpha}
      ^\frown \langle \beta \rangle} \mid \beta \in A\backslash(\max(\vec{\alpha})+1)\}$ forms a $\Delta$-system with root $u_{\vec{\alpha}}$;
    \end{enumerate}
    \item for all $m \leq n$ and all $\vec{\alpha}, \vec{\beta} \in [A]^m$,
    \begin{enumerate}
    \item $|u_{\vec{\alpha}}| = |u_{\vec{\beta}}|$, and
    \item if $\vec{\alpha}$ and $\vec{\beta}$ are aligned, then $u_{\vec{\alpha}}$
    and $u_{\vec{\beta}}$ are aligned.\footnote{
    Two finite sets of ordinals $u$ and $v$ are \emph{aligned} if $|u|=|v|$ and $|u\cap\alpha|=|v\cap\alpha|$ for all $\alpha\in u\cap v$.} 
    \end{enumerate}
    \end{enumerate}
    \item \label{q_cohere} $q_{\vec{\beta}} \restriction u_{\vec{\alpha}}
    = q_{\vec{\gamma}} \restriction u_{\vec{\alpha}}$ for all $\vec{\alpha} \in [A]^{<n}$ and $\vec{\beta}, \vec{\gamma} \in
    [A]^n$ such that $\vec{\alpha} \sqsubseteq \vec{\beta}$ and $\vec{\alpha}
    \sqsubseteq \vec{\gamma}$.
  \end{enumerate}
\end{lem} 

We now turn more directly to the argument that $V^{\Hbb_\kappa}\vDash\PH_n$ for all $n\geq 0$. 
\begin{proof}[Proof of Theorem \ref{PHfromHechlers}]
We have already noted that the theorem is true when $n=0$.
Fix $n > 0$ and an $\Hbb_\kappa$-name $\dot{f}$ and a $p\in\Hbb_\kappa$ forcing that $\dot{f}$ is a function from $\dot{\Omega}^{n+1}$ to $\omega$. Set $A_0:=\kappa\backslash(\max(\dom(p))+1)$.
For all $\vec{\alpha}\in [A_0]^{n+1}$ fix a $q_{\vec{\alpha}}\leq p$ in $D$ such that \begin{align}\label{whatsforced}
q_{\vec{\alpha}}\Vdash``\dot{z}_{\alpha_0}<\dots<\dot{z}_{\alpha_n}\text{ and }\dot{f}(\{\dot{z}_{\alpha_0},\dots,\dot{z}_{\alpha_n}\})=i(\vec{\alpha})\text{''}\end{align}
for some $i(\vec{\alpha})\in\omega$.
Apply the weak compactness of $\kappa$ to thin $A_0$ to a cofinal $A_1\subseteq\kappa$ such that $i(\vec{\alpha})$ equals some fixed $i$ for all $\vec{\alpha}\in [A_1]^{n+1}$.

Now apply Lemma \ref{HechlerDeltaSystem} to $\langle q_{\vec{\alpha}} \mid \vec{\alpha} \in [A_1]^{n+1} \rangle$ to find an unbounded $A \subseteq A_1$
  together with sets $\langle u_{\vec{\alpha}} \mid \vec{\alpha} \in [A]^{\leq n}
  \rangle$, a natural number $\ell$, and stems $\langle s_i \mid i < \ell \rangle$
  as in the statement of the lemma.
  Next, define conditions $\left\langle q_{\vec{\alpha}} ~ \middle| ~ \vec{\alpha}
  \in [A]^{\leq n} \right\rangle$
  as follows: for each $\vec{\alpha}$ in $[A]^{\leq n}$ let $\vec{\beta}$ be
  an element of $[A]^{n+1}$ such that $\vec{\alpha} \sqsubseteq \vec{\beta}$
  and let $q_{\vec{\alpha}} = q_{\vec{\beta}} \restriction u_{\vec{\alpha}}$.
  By item (\ref{q_cohere}) of Lemma \ref{HechlerDeltaSystem}, these definitions
  are independent of all of our choices of $(n+1)$-tuples $\vec{\beta}$.
  Moreover, the fact that $q_\emptyset=\bigcap_{\vec{\alpha}\in [A]^{n+1}}q_{\vec{\alpha}}$ implies that $q_\emptyset\leq q$.
  
 We claim that $q_\emptyset$ forces the existence of an $\Upsilon$ and $F$ as in Lemma \ref{extensionlemma}; in other words, $q_\emptyset$ forces that the conclusion of $\PH_n$ holds for the function $\dot{f}$.
 Since $\dot{f}$ named an arbitrary function from $\dot{\Omega}^{n+1}$ to $\omega$, showing this will conclude our proof.
 
We argue this claim by first partitioning $A$ into $n + 1$ disjoint and unbounded subsets $\{\Gamma_i\mid 1\leq i\leq n+1\}$.
 Let $\dot{B}$ be a $\Hbb_\kappa$-name for the set of
  $\alpha \in \Gamma_1$ such that $q_{\langle\alpha\rangle} \in \dot{G}$, where $\dot{G}$ is the canonical name for the $\Hbb_\kappa$-generic filter.
Observe that
  \[
    q_\emptyset \Vdash\text{``}\dot{B} \text{ is unbounded in } \kappa\text{''}.
  \]
  To see this, fix an $r\leq q_\emptyset$ and $\eta<\kappa$; it will suffice to find an $\alpha\in (\Gamma_1 \backslash\eta)$ such that $r$ and $q_{\langle \alpha\rangle}$ are compatible.
  To this end, note that as $\langle u_{\langle\alpha\rangle}\mid\alpha\in A\rangle$ forms a $\Delta$-system, there exists an $\alpha\in (\Gamma_1 \backslash\eta)$ with $u_{\langle\alpha\rangle}\backslash u_\emptyset\cap \dom(r)=\emptyset$.
  Since $q_{\langle \alpha\rangle}\restriction u_\emptyset=q_\emptyset$ and $q_\emptyset\geq r$, the conditions $q_{\langle \alpha\rangle}$ and $r$ are indeed compatible, as desired.
  
  This set $\dot{B}$ will index those Hechler reals comprising the $\Upsilon\subseteq\Omega$ upon which we'll define the function $F$ in the generic extension.
  This definition will depend on one further lemma; to state it, we adopt the following conventions.
\begin{defn}\label{61}
  For any nonempty $\tau$ in $[\kappa]^{<\omega}$, a \emph{subset-initial segment
  of $\tau$} is a sequence $\sigma_1 \subset \cdots \subset \sigma_m \subseteq \tau$
  such that
  \begin{itemize}
    \item $m \leq |\tau|$ and
    \item $|\sigma_i| = i$ for all $i$ with $1 \leq i \leq m$.
  \end{itemize}
  We write $\vec{\sigma} \vartriangleleft \tau$ to indicate that
  $\vec{\sigma}$ is a subset-initial segment of $\tau$.
  When ordinals $\alpha_{\sigma_i}$ have been associated to each element of a subset-initial $\vec{\sigma}\vartriangleleft \tau$ then we write $\vec{\alpha}[\vec{\sigma}]$ for the sequence $\langle \alpha_{\sigma_1},\dots,\alpha_{\sigma_m}\rangle$.
\end{defn}
The following appears (together with its proof) as Lemma 6.7 in \cite{SVHDL}.
\begin{lem}\label{strings}
   Let $\kappa$ be a weakly compact cardinal and fix $\tau\in [\kappa]^{n+1}$.
   The condition $q_\emptyset$ forces the following to hold
   in $V^{\Hbb_\kappa}$:
   whenever $1 < m \leq n+1$ and $\left\{\alpha_\sigma ~ \middle| ~ \sigma\in[\tau]^{<m}\textnormal{
   and }\sigma\neq\emptyset\right\}$ are such that
   \begin{enumerate}
     \item $\alpha_{\langle \gamma\rangle} = \gamma$ for all $\gamma \in \tau$,
     \item $\alpha_\rho<\alpha_\sigma$ whenever $\rho$ is a proper subset of $\sigma$,
     \item $\alpha_\sigma\in \Gamma_{|\sigma|}$ for all nonempty
     $\sigma\in[\tau]^{<m}$, and
     \item for any $1\leq\ell< m$ and subset-initial segment $\vec{\sigma}
     \vartriangleleft \tau$ of length $\ell$, we have $q_{\vec{\alpha}[\vec{\sigma}]}
     \in \dot{G}$ (in particular, $\eta \in \dot{B}$ for all $\eta \in \tau$),
   \end{enumerate}
   then it follows that there exists a collection $\left\{\alpha_\sigma ~ \middle| ~ \sigma\in[\tau]^m\right\}\subseteq\Gamma_m$ which satisfies
   \begin{enumerate}
   \setcounter{enumi}{4}
     \item $\alpha_\rho<\alpha_\sigma$ whenever $\rho$ is a proper subset of $\sigma$, and
     \item for any subset-initial segment $\vec{\sigma} \vartriangleleft \tau$
     of length $m$, we have $q_{\vec{\alpha}[\vec{\sigma}]} \in \dot{G}$.
   \end{enumerate}
 \end{lem}

We now fix an $\Hbb_\kappa$-generic filter $G$ containing $q_\emptyset$ and work in $V[G]$. We denote the interpretation of names therein simply by removing their dots and we let $\Upsilon=\{z_\alpha\mid\alpha\in B\}$.
Repeated application of Lemma \ref{strings} will determine a collection of ordinals $\{\alpha_\sigma\mid\sigma\in [B]^{\leq n+1}\backslash\{\emptyset\}\}$. For any $\xbf= (z_{\gamma_0},\dots,z_{\gamma_m} ) \in\Upsilon^{\leq n+1}$ then let $F(\xbf)=z_{\alpha_\sigma}$, where $\sigma=\{\gamma_0,\dots,\gamma_m\}$. Observe that $F$ satisfies
\begin{itemize}
\item $F(x)=x$ for all $x\in\Upsilon$, and
\item $F(\xbf) \leq F(\ybf)$ for any $\xbf\trleq \ybf$ in $\Upsilon^{\leq n+1}$.
\end{itemize}
The first point follows from item (1) of Lemma \ref{strings}, which entails that $F(z_\gamma)=z_{\alpha_{\langle \gamma\rangle}}=z_\gamma$ for all $\gamma\in B$.
The second point follows from items (2) and (5) of Lemma \ref{strings}, together with the fact that any $\rho\subset\sigma$ in  $[B]^{\leq n+1}\backslash\{\emptyset\}$ appear in some length-$(n+1)$ subset-initial $\vec{\sigma} \vartriangleleft \tau\subset B$ such that $q_{\vec{\alpha}[\vec{\sigma}]}\in G$.
By equation (\ref{whatsforced}), such conditions $q_{\vec{\alpha}[\vec{\sigma}]}$ will force the desired inequality, as well as the fact that $f\circ F^*$ takes the constant value $i$. 
\end{proof}

\section{The Partition Hypothesis and the additivity of \texorpdfstring{${\lim}^n$}{limn} for \texorpdfstring{$\Omega$}{O}-systems}
\label{Sect:PartHypAdd}
In the course of establishing the main results of \cite{strong_hom_add}, the authors
isolated a class of inverse systems of abelian groups indexed by $\Omega$.
These inverse systems are specified by a set of data called an \emph{$\Omega$-system}.
In this section, we recall the notion of an $\Omega$-system
and deduce from the Partition Hypothesis the additivity of derived limits associated to
their inverse systems.
\begin{defn}
An \emph{$\Omega$-system} is specified by an indexed collection $\Gcal=\{G_{n,k} \mid n,k \in \omega\}$
of finitely generated abelian groups together with compatible homomorphisms
$\pi_{n,j,k}: G_{n,k} \to G_{n,j}$ for
each $n$ and $j \leq k<\omega$.
Such data give rise to the following additional objects:
\begin{itemize}

\item 
For each $x \in \Om$, set 
$G_x := \bigoplus_{n=0}^\infty G_{n,x(n)}$
and
$\bar{G}_x := \prod_{n=0}^\infty G_{n,x(n)}$.
We regard $G_x$ as a subset of $\bar{G}_x$.

\item
For each $x \leq y$ in $\Om$ let
$\pi_{x,y}:\bar{G}_y \to \bar{G}_x$
denote the product homomorphism $\prod_{n=0}^\infty \pi_{n,x(n),y(n)}$. Write $\pi_{x,y}$ for these maps' restrictions $G_y\to G_x$ as well.
\end{itemize}
We write $\bar{\Gbf}$ and $\Gbf$ for the inverse systems over $\Omega$ whose terms are the groups $\bar{G}_x$ and $G_x$, respectively, and whose bonding maps are $\pi_{x,y}$.
We denote the $n\Th$ tower of groups in an $\Omega$-system by $\Gbf_n$; more precisely, $\Gbf_n$ is the inverse system indexed by $\omega$ with $(\Gbf_n)_k=G_{n,k}$.
An important point in what follows is that $\bigoplus_{n\in\omega}\Gbf_n\cong \Gbf$ in the category of pro-abelian groups.
 \end{defn}

We compute derived limits via the \emph{alternating chain complex}:

\begin{defn} \label{limscochaincomplex}
  Given an inverse system $\Xbf=(X_p,\pi_{p,q},\Pbb)$ over a quasi-lattice $\Pbb$,
  we say that $\Phi\in\prod_{p_0,\ldots,p_n}X_{\bigwedge p_i}$ is \emph{alternating} if for each $p_0,\ldots,p_n \in \Pbb$ and permutation
  $\sigma$ of $\{0,\ldots,n\}$, $\Phi(p_0,\ldots,p_n)=\operatorname{sgn}(\sigma)\Phi(p_{\sigma(0)},\ldots,\Phi(p_{\sigma(n)}))$. 
We define the cochain complex $C_{\alt}^\bullet(\Xbf)$ by 
\[C_{\alt}^n(\Xbf):=\left\{\Phi\in\prod_{p_0,\ldots,p_n}X_{\bigwedge_i p_i}\middle|\Phi\text{ is alternating}\right\}\]
with coboundary maps 
\[d^n \colon C_{\alt}^n(\Xbf) \to C_{\alt}^{n+1}(\Xbf)\]
given by
\[
d^n (\Phi)(\vec{p}\,) := \sum_{i=0}^{n+1}(-1)^i \pi_{\bigwedge \vec{p}, \bigwedge \vec{p}^{\,i}} (\Phi(\vec{p}^{\:i}))
\]
where $\vec{p}^{\:i} = (p_0,\ldots ,\widehat{p_i},\ldots ,p_{n+1})$ denotes the omission of the $i$th coordinate from $\vec{p} = (p_0,\ldots,p_{n+1})$.
Unless there is need for clarity, we will generally suppress the superscript on
$d^n$.
\end{defn}
\begin{defn}
Given an inverse system $\Xbf=(X_p,\pi_{p,q},\Pbb)$,
define ${\lim}^n\Xbf$ to be the $n\Th$ cohomology group
$\ker(d^{n})/\im(d^{n-1})$ of $C_{\alt}^\bullet(\Xbf)$ where $d^{-1}$ is the zero map. 
\end{defn}

Observe that for each $k$,
the inclusion map $\Gbf_k \to \bigoplus_i \Gbf_i \cong \Gbf$ induces inclusion maps
$C_{\alt}^n(\Gbf_k) \to C_{\alt}^n(\Gbf)$ for each $n$.
It is readily checked that, when passing to cohomology, these induce inclusions
${\lim}^n \Gbf_k \to {\lim}^n \Gbf$.
As ${\lim}^n$ is finitely additive, this in turn induces the inclusion map $\bigoplus_k {\lim}^n\,\Gbf_k\to{\lim}^n\, \Gbf$.
Our goal in this section will be to prove that $\PH_n$ implies that this induced
inclusion map $\bigoplus_k{\lim}^n\,\Gbf_k\to{\lim}^n\, \Gbf$ is an isomorphism.
We will start with the case $n=1$, which will be verified directly.

\begin{prop} \label{lim1_additive}
Suppose $\PH_1$. 
The induced inclusion map $\bigoplus_i{\lim}^1\,\Gbf_i\to{\lim}^1\, \Gbf$ is an isomorphism. 
\end{prop}

\begin{proof}
In order to see that the inclusion is a surjection,
fix a cocycle $\Phi\in C_{\alt}^1(\Gbf)$ and define $c:\Omega^{2}\to\omega$ by
\[c(x,y):=\min\left\{m\in\omega\,\middle|\, \Phi(x,y)\in \bigoplus_{i=0}^m G_{i,\min(x(i),y(i))}\right\}.\]
Let $F$ witness $\PH_1$ for $c$ and set $m_0$ as the constant value of $c\circ F^*$. 
By the comments made following the formulation of $\PH_n$, we may assume
without loss of generality that there is an $m \geq m_0$ such that $x\leq^m F(x)$ holds for all
$x \in \Omega$.
Set $\Psi(x)$ be the projection of $\Phi(x,F(x))$ to $\bigoplus_{i=m}^\infty G_{i,x(i)}$ and note $\Psi\in C_{\alt}^0(\Gbf)$.
Observe that, after projecting to $\bigoplus_{i=m}^\infty G_{i,\min(x(i),y(i))}$,
\[\begin{aligned}
0&=&&d\Phi(x,y,F(x,y))-d\Phi(x,F(x),F(x,y))+d\Phi(y,F(y),F(x,y))\\
&=&&\Phi(y,F(x,y))-\Phi(x,F(x,y))+\Phi(x,y)\\
&&&-\Phi(F(x),F(x,y))+\Phi(x,F(x,y))-\Phi(x,F(x))\\
&&&+\Phi(F(y),F(x,y))-\Phi(y,F(x,y))+\Phi(y,F(y))\\
&=&&\Phi(x,y)-d\Psi(x,y)-\Phi(F(x),F(x,y))+\Phi(F(y),F(x,y)),
\end{aligned}\]
where the first equality is using that $\Phi$ is a cocycle, the second is using the definition of $d$, and the third is cancelling like terms with opposite signs and the definition of $d\Psi$. 
Notice that, by our choice of $F$, each of $\Phi(F(x),F(x,y))$ and $\Phi(F(y),F(x,y))$ are supported on the first $m$ summands. 
Therefore, $[\Phi]$ is in the image of the map from $\bigoplus_{i=0}^{m-1}\lim{}^1\Gbf_i$ as $\Phi$ is equal, up to a coboundary, to a cocycle with each coordinate supported in the first $m$ coordinates.
\end{proof}

With a bit more care, the above proof generalizes to higher derived limits.

\begin{thm} \label{coherent_trivial}
Suppose $\PH_n$. 
The induced inclusion map $\bigoplus_i {\lim}^n\Gbf_i\to{\lim}^n\Gbf$ is an isomorphism. 
\end{thm}

\begin{proof}
As in the ${\lim}^1$ case, fix a cocycle $\Phi$ representing an element of ${\lim}^n\Gbf$. 
Let $F$ witness the conclusion of $\PH_n$ for the partition 
\[c_\Phi(x_0,\ldots,x_n)\mapsto\min\left\{m\in\omega\,\middle|\,\Phi(x_0,\ldots,x_n)\in\bigoplus_{i=0}^mG_{i,\bigwedge x_j(i)}\right\}.\]
As noted following the formulation of $\PH_n$, we may choose $F$ so that for some fixed
$k$, $F$ is $(n+1)$-cofinal with respect to $\leq^k$ and $c\circ F^*$ is constant with value at most $k$.
Let $\Free(X)$ denote the free abelian group over $X$ and denote the basis element corresponding to $x$ by $e(x)$ (inside the argument of $e$, for readability we may omit the brackets demarcating a collection $x$).
Organizing our argument are several interrelated formal expressions.
For $1 \leq s\leq n$ and $\rho\in\Om^{s}$,
\[A_s(\rho)\text{ denotes an element of }\Free(\{x\mid \rho_0 \leq^k x\leq F(\rho)\}^{s+1}).\]
For $\tau\in\Om^{s+1}$,
\[C_s(\tau)\text{ denotes an element of }\Free(\{x\mid \tau_0 \leq^k x \leq F(\tau)\}^{s+1})\text{, and}\]
\[S_s(\tau)\text{ denotes an element of }\Free(\{x\mid \tau_0 \leq^k x \leq F(\tau)\}^{s+1}).\]
For any $\rho$ as above, we define
\[d\colon \Free(\{x\mid \rho_0 \leq^k x\}^{s+1})\to \Free(\{x\mid \rho_0 \leq^k x\}^{s})\]
by $de(x_0,\ldots ,x_s):=\sum_i(-1)^ie(x_0,\ldots ,\widehat{x_i},\ldots ,x_s)$.
Similarly, for any $y\in\Om$ we define an operation $x \mapsto x*y$ from
$\Free(\{x\mid\rho_0 \leq^k x \leq y\}^{s})$  to $\Free(\{x\mid\rho_0 \leq^k x \leq y\}^{s+1})$ 
by setting $e(x_0,\ldots ,x_{s-1})*y=e(x_0,\ldots ,x_{s-1},y)$.
Together, these two operations satisfy the relation 
\[d(x*y)=d(x)*y+(-1)^s x.\]
The idea of these expressions will be the following: $A_s$ is the stage $s$ approximation to a $\Psi$ satisfying $d\Psi-\Phi=0$ after some fixed finite number of coordinates, $S_s$ is (up to sign) the coboundary of $A_{s+1}$, and $C_s$ is an error term recording the difference between the coboundary of $A_s$ and $\Phi$. 
Note that for each $\rho$, $\Phi$ determines a map $\Ecal_{\Phi}^\rho:\Free(\{x\mid \bigwedge\rho \leq^k x\}^{n+1})\to  G_{\bigwedge\rho}$, namely the map given by sending $e(\sigma)$ to $\Phi(\sigma)$, then projecting to $G_{\bigwedge\rho}$ using zero maps in the first $k$ coordinates and the $\Omega$-system's bonding maps on all subsequent coordinates.
For the base case $A_1$, we set
\[A_1(\rho)=e(\rho,F(\rho))\]
as in the proof of Proposition \ref{lim1_additive}.
In general, we let
\begin{align*} C_s(\tau) & =e(\tau)-\sum_{i<s+1}(-1)^iA_s(\tau^i),\\
S_s(\tau) & =d\left(C_s(\tau)*F(\tau)\right),\textnormal{ and}\\
A_{s+1}(\tau) & =(-1)^{s+1}C_s(\tau)*F(\tau).\end{align*}
The following lemma is the key point ensuring that our approximations converge to a $\Psi$ as above. 
If $\tau \in \Omega^s$, it will be convenient to write $\tau^{\bb{s}}$ for all elements of $\Omega^{\bb{s}}$ whose last coordinate is $\tau$.
\begin{lem}
\label{lem:13}
For all $s\leq n$ 
and all $\tau \in \Omega^{s+1}$, 
$C_s(\tau)$ is of the form $(-1)^{s+1}S_s(\tau)$ plus terms of the form $e(F^*(\vec\sigma))$ for $\vec\sigma\in\tau^{\bb{s+1}}$. 
\end{lem}

The proof of this lemma will complete our proof of the theorem.
For by this decomposition of $C_n(\tau)$, together with our hypotheses on $F$ and the fact that $\Ecal_{\Phi}^\rho(S_n(\tau))=0$ for all $\tau\in\Omega^{n+1}$ and all $\rho\in\Omega^{n+1}$ (since $\Ecal_\Phi^\rho(S_n(\tau))$ is a sum of coboundary terms of $\Phi$ and $\Phi$ is a cocycle), the family $\Psi$ defined by $\Psi(\rho)=\Ecal_\Phi^\rho(A_n(\rho))$ satisfies $d\Psi(\tau)=\Phi(\tau)$ once projected onto $\bigoplus_{i=k+1}^\infty G_{i,\bigwedge\tau(i)}$; in particular, $[\Phi]$ is in the image of the map from $\bigoplus_{i=0}^k\lim^n\Gbf_i$.

\begin{proof}[Proof of Lemma \ref{lem:13}]
Observe that $S_s$ can be rewritten as follows: 
\[\begin{aligned}
&S_s(\tau)&&=d(C_s(\tau))*F(\tau)+(-1)^{s+1}C_s(\tau)\\
&&&=d(e(\tau))*F(\tau)-d\left(\sum_{i<s+1}(-1)^iA_s(\tau^i)\right)*F(\tau)+(-1)^{s+1}C_s(\tau)\\
(*) &&&=\sum_{i<s+1}(-1)^ie(\tau^i,F(\tau))-\sum_{i<s+1}(-1)^i\sum_{j<s+1}[A_s(\tau^i)*F(\tau)]^j\\
&&&\quad+(-1)^{s+1}C_s(\tau).
\end{aligned}\]

The proof proceeds by induction on $s$ to show that $(*)$ is a sum of terms of the form $e(F^*(\vec\sigma))$ for $\vec\sigma\in\tau^{\bb{ s+1}}$. 
The $s=1$ case was given already in the proof of Proposition \ref{lim1_additive}, noting that $S_1(x,y)=de(x,y,F(x,y))-de(x,F(x),F(x,y))+de(y,F(y),F(x,y))$, the expression used in that proof.

Now assume that $s>1$ and the induction hypothesis holds for $s-1$. 
Note that, by the definition of $S_s$, for $s>2$ $(*)$ is equal to
\[\sum_{i<s+1}(-1)^ie(\tau^i,F(\tau))-(-1)^s\sum_{i<s+1}(-1)^iS_{s-1}(\tau^i)*F(\tau);\]
this is perhaps more apparent from the line above in our earlier calculation plus the observation that 
\[S_s(\tau)=(-1)^{s+1}d(A_{s+1}(\tau)).\]
Note that the induction hypothesis implies that $(-1)^sS_{s-1}(\tau^i)*F(\tau)$ is of the form $C_{s-1}(\tau^i)*F(\tau)$ plus terms of the form $e(F^*(\vec\sigma))$ with $\vec\sigma\in\tau^{\bb{s+1}}$ since $e(F^*(\vec\rho))*F(\tau)=e(F^*(\vec\rho^\frown\langle \tau\rangle))$ which is of the appropriate form. 
Thus $(*)$ reduces to terms of the form $e(F^*(\vec\sigma_1))$ plus

\adjustbox{max width=\textwidth}{$\begin{aligned}
& \sum_{i<s+1}(-1)^{i} e\left(\tau^{i}, F({\tau})\right)-\sum_{i<s+1}(-1)^{i}\left[C_{s-1}\left(\tau^{i}\right) * F({\tau})\right] \\
=& \sum_{i<s+1}(-1)^{i} e\left(\tau^{i}, F({\tau})\right)-\sum_{i<s+1}(-1)^{i}\left[e\left(\tau^{i}, F({\tau})\right)-\sum_{j<s}(-1)^{j} A_{s-1}\left(\left(\tau^{i}\right)^{j}\right) * F({\tau})\right] \\
=& \sum_{i<s+1} \sum_{j<s}(-1)^{i+j} A_{s-1}\left(\left(\tau^{i}\right)^{j}\right) * F({\tau}).
\end{aligned}$}

The key observation is that for any term with $i\leq j$, the term with $i'=j+1$ and $j'=i$ is the same but with opposite sign, thus showing that 
\[\sum_{i<s+1} \sum_{j<s}(-1)^{i+j} {A}_{s-1}\left(\left(\tau^{i}\right)^{j}\right) * F({\tau})=0,\]
which completes the proof.
\end{proof}
\end{proof}

The next theorem summarizes this section's results; for further details of the second implication, see \cite{strong_hom_add}.
It is shown in \cite{strong_hom_add} that (\ref{add}) implies (\ref{triv}), although the converse is unclear. 
\begin{thm}
\label{thm:summary}
If $\PH_n$ holds for all $n\in\omega$, then each of the following holds as well.
\begin{enumerate}
    \item For any $\Om$-system $\Gcal$ and $n\in\omega$, the inclusion map
    \[\bigoplus_{k\in\omega}{\lim}^n\,\Gcal_k\to{\lim}^n\,\Gbf\]
    is an isomorphism.
    Put differently, in the category of pro-abelian groups, for any $n\in\omega$ and countable discrete diagram of inverse sequences of finitely-generated abelian groups, the functors ${\lim}^n$ and $\colim$ commute.
    \label{add}
    \item Strong homology is additive and has compact supports on the class of locally compact separable metric spaces.
    \item Every $n$-coherent family of functions is $n$-trivial.
    \label{triv}
\end{enumerate}
\end{thm}

\section{Generalizing the Partition Hypothesis}
As indicated, notions of $n$-cofinal functions and associated partition hypotheses make sense on any directed partial order.
Any systematic treatment of these generalizations evidently falls beyond the scope of the present work; in this section, however, we do record a few basic observations about
\begin{itemize}

\item partition hypotheses for arbitrary products of partial orders, and
\item partition hypotheses for ordinals.

\end{itemize}
Our interest in the latter is motivated in part by connections between the behaviors of the ${\lim}^n$ functors on $\Omega$-indexed inverse systems and their behavior on $\dfrak$-indexed inverse systems, connections operative in the model $V^{\Hbb_\kappa}$ of Section \ref{Sect:forcing}, for example.
(Here $\dfrak$ is the minimum cardinality of a cofinal subset of $\Omega$.)
These hypotheses appear also to be of some interest in their own right, particularly on the ``small'' cardinals $\omega_n$, where they raise multiple substantial questions which we record in our conclusion.
We conclude this section with a discussion of partition hypotheses within the framework of simplicial sets, which affords us one further, and surprisingly natural, framing of the principles $\PH_n(\Pbb,\lambda)$.

\label{Sect:Generalizing}
\subsection*{Generalizing to arbitrary quasi-orders}

In this section, we will return to the full generality of the partition hypotheses $\PH_n(\Pbb,\lambda)$ introduced in Section \ref{Sect:PH}.
One value of this generalization is that it allows for a comparison of partition hypotheses on various quasi-orders.

\begin{lem} \label{reduction}
Suppose that $\Pbb,\mathbb{Q}$ are directed quasi-orders, $\PH_n(\Pbb,\lambda)$ holds, and $f\colon\Pbb\to\mathbb{Q}$ is a monotone map with cofinal image. 
Then $\PH_n(\mathbb{Q},\lambda)$ holds. 
\end{lem}
\begin{proof}
Let $g\colon\mathbb{Q}\to\Pbb$ be such that $f(g(q))\geq q$ for each $q\in\mathbb{Q}$. 
Given $c\colon \mathbb{Q}^{[n+1]}\to\lambda$, let $\widetilde{c}\colon\Pbb^{[n+1]}\to\lambda$ be given by $\widetilde{c}(p_0,\ldots,p_n)=c(f(p_0),\ldots,f(p_n))$. 
Let $F\colon\mathbb{P}^{\leq n+1}\to\mathbb{P}$ be such that $\widetilde{c}\circ F^*$ is constant. 
Define $\bar{F}\colon\mathbb{Q}^{\leq n+1}\to\mathbb{Q}$ by $\bar{F}(q_0,\ldots,q_m)=f(F(g(q_0),\ldots,g(q_n)))$. 
Then $\bar{F}$ is $(n+1)$-cofinal by our hypotheses on $f$, $g$. 
Moreover, $c\circ\bar{F}$ is constant by hypotheses on $F$ and the definition of $\widetilde{c}$. 
\end{proof}

Recall that for two directed quasi-orders $\mathbb{P}$ and $\mathbb{Q}$, we say that $\mathbb{Q}$ is \emph{Tukey reducible} to $\mathbb{P}$ and write $\mathbb{Q}\leq_T\mathbb{P}$, if there is a map $f:\mathbb{P}\rightarrow \mathbb{Q}$ mapping cofinal subsets of $\mathbb{P}$ to cofinal subsets of $\mathbb{Q}$, or equivalently, if there is a map $g: \mathbb{Q}\rightarrow \mathbb{P}$ such that for every $p\in \mathbb{P}$ the set $\{q\in\mathbb{Q}: g(q)\leq p\}$ has an upper bound in
$\mathbb{Q}$. 
Notice that if $\mathbb{P}$ is a cofinal subset of $\mathbb{Q}$, then 
$\mathbb{P} \leq_T \mathbb{Q}$ and $\mathbb{Q} \leq_T \mathbb{P}$.
It is natural, particularly in lieu of Lemma \ref{reduction}, to wonder how partition hypotheses may or may not transmit along Tukey reductions; we record this question in our conclusion.

\begin{rem}
As noted, $\PH_n$ is an abbreviation of $\PH_n(\Om,\omega)$,
and in Section \ref{Sect:MeasPartHyp} we have seen that a measurable version of $\PH_n$ could hold.
What we would like to point out here is that it makes sense to consider a measurable version of
$\PH_n(\Pbb,\omega)$ for any specific Borel quasi-order $\Pbb$. The Tukey hierarchy of 
Borel directed orders is a relatively well-developed theory and the quasi-order $\Om$
has a special place in it.
However there are other Borel quasi-orders $\Pbb$ that also have special places
in this hierarchy and to which, moreover, $\Om$ Tukey reduces
(even to the extent that there is a Borel monotone map $f: \Pbb\rightarrow \Om$ with cofinal range in $\Om$),
and therefore,
for which the corresponding hypothesis $\PH_n(\Pbb,\omega)$ is stronger than $\PH_n$. 
One such special Borel directed quasi-order is the Banach lattice $\ell_1$.
The proof in \cite{To89} shows that, assuming $\OCA + \add (\ell^1) > \aleph_1$, if $\phi_x$ $(x \in \ell^1)$ is a family of functions $\phi_x:D_x \to \Zbb$
which cohere mod finite and such that if $x \leq y$ then $D_x \subseteq D_y$, then $\{\phi_x \mid x \in \ell^1\}$ is trivial.
If $\{\phi_x \mid x \in \ell^1\}$ is Borel, then the assertion that $\{\phi_x \mid x \in \ell^1\}$ is trivial
is equivalent to a $\PCA$-sentence and hence absolute.
Since $\OCA + \add (\ell^1) > \aleph_1$ can be forced over any model of ZFC, it follows from Shoenfield's absoluteness theorem \cite{Sho} 
that ZFC proves that all Borel coherent families indexed by $\ell^1$ (in the above sense) are trivial.
It is natural to ask if this argument can be generalized to higher dimensions and if the Borel version of
$\PH_n(\ell_1)$ for $n \geq 1$ is true.
See \cite{SoTo} and the references therein for more information on Borel quasi-orders and Tukey reductions. 
\end{rem}

We now record implications of partition hypotheses for the additivity of derived limits generalizing those recorded in Section \ref{Sect:PartHypAdd}. 
The next definition generalizes the notion of $\leq^k$. 

\begin{defn}
Given $\langle\Pbb_i\mid i\in I\rangle$ directed posets, $S\subseteq I$, and $x,y\in\prod_{i\in I}\Pbb_i$, we say that $x\leq^Sy$ if $x(i)\leq_{\Pbb_i} y(i)$ whenever $i\not\in S$.
We write $x\leq^*y$ if $x\leq^Sy$ for some finite $S\subseteq I$. 
\end{defn}

Theorem \ref{coherent_trivial} admits a natural generalization given by the following, with the same proof.

\begin{thm}
Let $\langle\Pbb_i\rangle_{i\in I}$ be posets and suppose that whenever $c\colon\left(\prod_{i\in I}\Pbb_i^{n+1}\right)\to |I|$ is a function, there is an $F$ which is $(n+1)$-cofinal with respect to $\leq^S$ for some finite $S$ such that $c\circ F^*$ is constant. 
Then for every collection of inverse systems $\left(\Gbf_i\right)_{i\in I}$ with each $\Gbf_i$ indexed over the corresponding $\Pbb_i$, the inclusion map 
\begin{align}
\label{eq:general_additivity}\bigoplus_{i\in I}{\lim}^n\,\Gbf_i\to {\lim}^n\,\bigoplus_{i\in I}\Gbf_i
\end{align}
is an isomorphism. 
\end{thm}

We note the following generalization of the observation which was made after defining
$\PH_n$.
Note that the index set in this case is $\omega$ rather than $I$: although ensuring there is a finite $S$ such that $F(\xbf)\leq^S F(\ybf)$ for $\xbf\trianglelefteq\ybf$ is unchanged, constructing a finite $S$ with $x\leq^SF(x)$ for each $x$ seems to require a countable index set. 

\begin{prop}
Suppose $\langle \Pbb_i\mid i<\omega\rangle$ are given and $\PH_n(\prod_{i\in \omega}\Pbb_i,\lambda)$ holds, where $\mathbb{P}=\prod_i\Pbb_i$ is given the ordering $\leq^*$. 
Then whenever $c\colon\Pbb^{n+1}\to\lambda$, there is an $F\colon\prod_i\Pbb_i^{[n+1]}\to\Pbb$ which is $(n+1)$-cofinal with respect to $\leq^k$ for some $k<\omega$.  
\end{prop}

Finally, it's not difficult to see that the principle $\PH_1(\prod_{i\in\omega}\omega_1,\omega)$ is false; see \cite{Pra} for related failures of equation \ref{eq:general_additivity} to be an isomorphism.
The results of this section should help to demarcate which additivity relations are consequences of, or consistent with, or outright inconsistent with, the ZFC axioms.
\subsection*{The Partition Hypothesis on the ordinals}

For any ordinal $\varepsilon$ and $n \in \omega$, let $\PH_n(\varepsilon)$ denote $\PH_n(\varepsilon,\omega)$.
A main goal of this section is to prove the following result.\footnote{See also \cite[Section 4.1]{LHZ} for an alternate recent proof of this theorem.}
\begin{thm}
\label{theorem:partition_failure_ordinals}
For all $n\in\omega$ the partition hypothesis $\PH_n(\omega_n)$ is false.
\end{thm}

It will be valuable to precede the proof of Theorem \ref{theorem:partition_failure_ordinals} with the following lemma. Call an $n$-cofinal function $F:\varepsilon^{\leq n}\to\varepsilon$ \emph{strictly increasing} if $F(\mathbf{x})<F(\mathbf{y})$ for all $\mathbf{x}\triangleleft\mathbf{y}$ in $\varepsilon^{\leq n}$.

\begin{lem}\label{strictly_increasing}
Let $\varepsilon$ be a limit ordinal. The partition hypothesis $\mathrm{PH}_n(\varepsilon)$ holds if and only if for every $c:\varepsilon^{n+1}\to\omega$ there is a strictly increasing $(n+1)$-cofinal $F:\varepsilon^{\leq n+1} \to \varepsilon$ such
that $c \circ F^*$ is constant.
\end{lem}

\begin{proof}
For the nontrivial implication, fix a $c$ as in the statement of the lemma.
Define the coloring $b:\varepsilon^{n+1}\to\{0,1\}$ by $b(\mathbf{x})=1$ if and only if $\mathbf{x}=(x_0,\dots,x_n)$ is a strictly increasing sequence of ordinals, and define $d:\varepsilon^{n+1}\to\omega\times\{0,1\}$ by $d(\mathbf{x})=(c(\mathbf{x}),b(\mathbf{x}))$. Our assumption $\mathrm{PH}_n(\varepsilon)$ implies that there exists an $(n+1)$-cofinal $F$ for which $d\circ F^*$ is constant; this implies in turn that $F$ is a strictly increasing $(n+1)$-cofinal function for which $c\circ F^*$ is constant.
\end{proof}

Note that the above argument applies \emph{mutatis mutandis} to any quasi-order without maximal elements.
Note also that to define an $F$ as above, it suffices to define one on the strictly increasing elements of $\varepsilon^{\leq n+1}$, an observation we will sometimes implicitly apply below.

To motivate our proof of Theorem \ref{theorem:partition_failure_ordinals}, we turn first to the cases of $n=0,1$, and $2$.
Clearly the coloring $c_0: i\mapsto i$ witnesses the failure of $\PH_0(\omega)$.
For the case of $\PH_1(\omega_1)$, fix injections $f_\beta:\beta\to\omega$ for each $\beta<\omega_1$ and let $c_1(\alpha,\beta)=c_1(\beta,\alpha)=f_\beta(\alpha)$ for all $\alpha<\beta<\omega_1$.
For any strictly increasing $2$-cofinal $F:\omega_1^{\leq 2}\to\omega_1$ there exist $\alpha<\beta<\omega_1$ with $F(\alpha) < F(\beta) < F(\alpha,\beta)$.
It then follows immediately from the fact that $f_{F(\alpha,\beta)}$ is injective that
$$c_1\circ F^*(\alpha,(\alpha,\beta))\neq c_1\circ F^*(\beta,(\alpha,\beta)).$$
For the case of $\PH_2(\omega_2)$, begin by fixing an injection $g_\gamma:\gamma\to\omega_1$ for each $\gamma<\omega_2$.
Define
$$c_2(\alpha,\beta,\gamma):=c_1(g_\gamma(\alpha),g_\gamma(\beta))$$
for all $\alpha<\beta<\gamma<\omega_2$.
Fix an arbitrary strictly increasing $3$-cofinal $F:\omega_2^{\leq 3}\to\omega_2$; we will show that $c_2\circ F^*$ is not constant.
Note that since $F$ is $3$-cofinal, there exist $\alpha<\beta<\gamma<\omega_2$ such that $$F(\alpha) < F(\beta) < F(\alpha,\beta) < F(\gamma) < F(\alpha,\gamma)\leq F(\beta,\gamma) < F(\alpha,\beta,\gamma).$$
(It is only perhaps not obvious how to arrange the penultimate inequality; to see this, fix a sequence $a=\langle\alpha_i\mid i\in\omega\rangle$ with $F(\alpha_i)<\alpha_j$ for any $i<j<\omega$, and a $\gamma>\mathrm{sup}_{i<j<\omega}\,F(\alpha_i,\alpha_j)$.
Observe now that if there did not exist an $i<j$ with $F(\alpha_i,\gamma)\leq F(\alpha_j,\gamma)$ then $\langle F(\alpha_i,\gamma)\mid i\in\omega\rangle$ would determine an infinite decreasing sequence of ordinals, a contradiction.)
Let $\delta$ denote $F(\alpha,\beta,\gamma)$ and let $A$ collect the ordinals listed above with the exception of $\delta$.
Define an edge-relation on $A$ by $\{F(\mathbf{x}),F(\mathbf{y})\}\in E_A$ if and only if $\mathbf{x}\triangleleft\mathbf{y}$. Note that the graph $(A,E_A)$ contains a cycle and, in consequence, that its $g_\delta$-image does as well.
Therefore there exist $\mathbf{x}_0,\mathbf{x}_1,\mathbf{y}_0,\mathbf{y}_1$ with
\begin{enumerate}[(i)]
\item \label{g_ineq} $g_\delta(F(\mathbf{x}_0))< g_\delta(F(\mathbf{x}_1))< g_\delta(F(\mathbf{y}_0))=g_\delta(F(\mathbf{y}_1))$, and
\item \label{E_A_cond} $\{F(\mathbf{x}_0),F(\mathbf{y}_0)\}$ and $\{F(\mathbf{x}_1),F(\mathbf{y}_1)\}$ both in $E_A$.
\end{enumerate}
Item (\ref{g_ineq}) implies that
\begin{align*} & c_2(F(\mathbf{x}_0),F(\mathbf{y}_0),\delta)=c_1(g_\delta(F(\mathbf{x}_0)),g_\delta(F(\mathbf{y}_0))) \\ \neq\, & c_2(F(\mathbf{x}_1),F(\mathbf{y}_1),\delta)=c_1(g_\delta(F(\mathbf{x}_1)),g_\delta(F(\mathbf{y}_1)))
\end{align*}
and item (\ref{E_A_cond}) implies that the arguments of the two $c_2$ terms above both fall in the image of $F^*$; this concludes the argument of the case $n=2$.

All the ideas of the proof of Theorem \ref{theorem:partition_failure_ordinals} are essentially present in this sequence.
Needed for its more general argument, however, is a vocabulary for higher-dimensional analogs of graphs and the higher-dimensional cycles within them.

\begin{defn}
\label{simplicial_definitions}
An \emph{abstract simplicial complex} $Y$ on a set $X$ is a family of nonempty finite subsets of $X$ which is closed under the taking of nonempty subsets.
The \emph{$n$-faces of $Y$} are its size-$(n+1)$ subsets, and we write $[Y]_n$ for their collection.
The \emph{dimension of $Y$} is $\dim(Y):=\sup\{n\mid [Y]_n\neq\emptyset\}$, and $Y$ is \emph{pure} if its set of $\subset$-maximal faces coincides with $[Y]_{\dim(Y)}$.
If $|a|=n+1$ then we write $\Delta_n(a)$ for the abstract $n$-simplex determined by $a$; this is simply $\Pscr(a)\backslash\{\emptyset\}$ viewed as an abstract simplicial complex.
Its \emph{boundary} $\Pscr(a)\backslash\{\emptyset,a\}$ is denoted $\partial\Delta_n(a)$.
For any $v\not\in X$ let $Y*v$ denote the \emph{cone over $Y$ by $v$}, namely $Y\cup\{b\cup\{v\}\mid b\in Y\cup\{\emptyset\}\}$.
If $v$ is an ordinal then $Y<v$ will mean that the vertex-set $X$ underlying $Y$ is a subset of $v$.
The \emph{barycentric subdivision $\sd (Y)$ of $Y$} is the abstract simplicial complex on $Y$ whose faces are the nonempty subsets of $Y$ which are linearly ordered by inclusion.

For any $Y$ as above we may consider its \emph{simplicial homology with coefficients in $\Zbb/2\Zbb$}.
This is the homology of the chain complex $\Ccal(Y)$ of the free $\Zbb/2\Zbb$-modules $C_n(Y)$ which are generated by the $n$-faces $a$ of $Y$; writing $\langle a\rangle$ for the generator associated to $a$, the boundary maps $d_n:C_n(Y)\to C_{n-1}(Y)$ of $\Ccal(Y)$ are those induced by the assignments $\langle a\rangle\mapsto\sum_{b\in [\partial\Delta_n(a)]_{n-1}}\langle b\rangle$. A \emph{homological $n$-cycle} is an $x\in C_n(Y)$ such that $d_n(x)=0$.

The following definitions and proposition are adapted from \cite[Defs. 4.1, 4.2, Prop. 5.1]{Connon}.
A sequence $a_0,\dots,a_k$ of elements of $[Y]_n$ is an \emph{$n$-path} if $|a_i\cap a_{i+1}|=n$ for all $i<k$.
If there exists an $n$-path between any two of its $n$-faces then $Y$ is \emph{$n$-path-connected}.
The maximal $n$-path-connected subcomplexes of $Y$ are its \emph{$n$-path components}.
For $n>0$, an \emph{$n$-cycle} is a pure $n$-dimensional simplicial complex $Y$ such that
\begin{itemize}
\item $Y$ is $n$-path connected, and
\item every $(n-1)$-face of $Y$ belongs to an even number of $n$-faces of $Y$.
\end{itemize}
For $n=0$, an $n$-cycle is simply a set of some even number of vertices.
\end{defn}
\begin{prop}
\label{prop:connon}
If $Y$ is a finite $n$-cycle with $n$-faces $a_0,\dots,a_k$ then $\sum_{i=0}^{k} \langle a_i\rangle$ is a homological $n$-cycle.
Conversely, if $\sum_{i=0}^{k} \langle a_i\rangle$ is a homological $n$-cycle then the $n$-path components of the simplicial complex generated by $\{a_0,\dots,a_k\}$ are $n$-cycles.
\end{prop}

Write $[\varepsilon]^n$ for the collection of $n$-element subsets of $\varepsilon$.
Observe that the colorings $c_n:\omega_n^{n+1}\to\omega$ partially defined above may each be regarded as deriving from a coloring $\tilde{c}_n:[\omega_n]^{n+1}\to\omega$.
Put differently, the ordering of the elements of the argument of $c_n$ was immaterial to, and even a distraction from, our proof of the cases of $n=0,1$ and $2$.
Similarly, if $F:\omega_n^{\leq n+1}\to\omega_n$ is strictly increasing then $F^*$ may be viewed as outputting elements of $[\omega_n]^{n+1}$, since its output consists simply in increasing enumerations of such elements; below, we will tend to identify finite sets of ordinals with their increasing enumerations without further comment.

In other words, in the language of Definition \ref{simplicial_definitions} the crux of the argument of the $n=1$ case above is the fact that $\tilde{c}_1\restriction [s*F(\alpha,\beta)]_1$ is non-constant for any $0$-cycle $s < F(\alpha,\beta)$ and that $\{F(\alpha),F(\beta)\}$ is just such a $0$-cycle $s$.
The crux of the argument of the $n=2$ case is the fact that $\tilde{c}_2\restriction [t*F(\alpha,\beta,\gamma)]_2$ is non-constant for any $1$-cycle $t<F(\alpha,\beta,\gamma)$ (for the reason that any such $t$ contains in turn an $[s*x]_1$ for some $0$-cycle $s<x$) and that the collection of edges $E_A$ contains such a $1$-cycle $t$.
We will now define the colorings $\tilde{c}_n:[\omega_n]^{n+1}\to\omega$ more generally, show that each is non-constant on the $n$-faces of cones over $(n-1)$-cycles, and prove Theorem \ref{theorem:partition_failure_ordinals} by showing that for any strictly increasing $(n+1)$-cofinal $F$ the $F^{*}$-image of $\omega_n\bb{n+1}$ must contain such cones.

The colorings $\tilde{c}_n:[\omega_n]^{n+1}\to\omega$ are defined by recursion on $n\in\omega$; the base cases of $n\leq 2$ were described above.  Define $\tilde{c}_{n+1}$ from $\tilde{c}_n$ by fixing injections $h_\beta:\beta\to\omega_n$ for each $\beta\in\omega_{n+1}$ and letting $\tilde{c}_{n+1}(\{\alpha_0,\dots,\alpha_n,\beta\})=\tilde{c}_n(\{h_\beta(\alpha_0),\dots,h_\beta(\alpha_n)\})$ for all $\alpha_0<\dots<\alpha_n<\beta<\omega_{n+1}$.
\begin{lem}
\label{lemma:coloring_cones}
For any $n>0$ and $(n-1)$-cycle $t<\delta<\omega_n$, the coloring $\tilde{c}_n$ is non-constant on $[t*\delta]_n$.
\end{lem}
\begin{proof}
The argument is by induction on $n$; the cases of $n=1$ and $n=2$ were established above.
Therefore assume that the lemma holds for some $n=m$; we will show that it holds for $n=m+1$ as well.
To that end, fix a $t<\delta<\omega_{m+1}$ as in the statement of the lemma, let $\gamma=h^{-1}_\delta(\max\{h_\delta(\alpha)\mid\alpha\in\bigcup t\})$, and let $\{\sigma_i\mid i<k\}$ enumerate the $m$-faces of $t$ containing $\gamma$.
By \cite[Prop. 4.3]{Connon}, $\{\sigma_i\backslash\{\gamma\}\mid i<k\}$ contains an $(m-1)$-cycle $s$, and hence the $h_\delta$-image of $s$ (which we will write as $h_\delta(s)$) is an $(m-1)$-cycle below $\gamma$.
Since $\tilde{c}_{m+1}(\sigma_i\cup\{\delta\})=\tilde{c}_m(h_\delta``\sigma_i)$ for any $i<k$, the range of $\tilde{c}_m\restriction [h_\delta(s)*\gamma]_m$ is contained in the range of $\tilde{c}_{m+1}\restriction [t*\delta]_{m+1}$; together with our induction hypothesis, this implies that the latter function is non-constant, as claimed.
\end{proof}
\begin{proof}[Proof of Theorem \ref{theorem:partition_failure_ordinals}]
We will show that for any $n>0$ and strictly increasing $(n+1)$-cofinal $F$ the $F^*$-image of $\omega_n\bb{n+1}$ contains a cone $[t*\delta]_n$ for some $(n-1)$-cycle $t<\delta<\omega_n$.
For any sequence $\sigma\in\omega_n^{n+1}$ of distinct ordinals, let $c_n(\sigma)$ be $\tilde{c}_n$ of the underlying set of $\sigma$.
By Lemmas \ref{strictly_increasing} and \ref{lemma:coloring_cones}, the coloring $c_n$ will witness the failure of $\mathrm{PH}_n(\omega_n)$.

To this end, fix such an $n$ and $F$; by our work above, we may assume that $n>2$.
For some $a\in [\omega_n]^{n+1}$, the aforementioned $(n-1)$-cycle $t$ and cone-point $\delta$ will derive from the $(n-1)$-cycle
$\sd (\partial\Delta_n(a))$ and cone-point $a$, respectively, in the cone-decomposition of $\sd(\Delta_n(a))$ as $\sd(\partial\Delta_n(a))*a$. In particular, $\delta$ will equal $F(a)$.
Care is needed in the choice of $a$ simply to ensure that $F$ maps enough of the vertices of $\sd(\partial\Delta_n(a))$ to distinct ordinals that its $F$-image remains $(n-1)$-cyclic.
To sum up, our argument will consist in two steps:
\begin{enumerate}

\item verifying that $\sd(\partial\Delta_n(a))$ is an $(n-1)$-cycle for any $a\in [\omega_n]^{n+1}$, and

\item showing that for well-chosen $a$ this implies that the $F$-image of this cycle indeed contains an $(n-1)$-cycle $t$, as desired.

  \end{enumerate}

\begin{claim} \label{claim1}
  For any $a\in [\omega_n]^{n+1}$, $\sd(\partial\Delta_n(a))$ is an $(n-1)$-cycle.
\end{claim}

\begin{proof}
Any $(n-2)$-face $b$ of $\sd(\partial\Delta_n(a))$ consists of a chain $\sigma_0 \subset \ldots \subset \sigma_{n-2}$ of elements of $[a]^{\leq n}$.
Some $j\in\{1,\dots,n\}$ is the cardinality of none of the elements of this chain and the $(n-1)$-faces $c$ of $\sd(\partial\Delta_n(a))$ containing $b$ consist precisely of the expansions of this chain by the addition to it of a $\sigma$ of length $j$.
But for any $\sigma'\subset\sigma''\subseteq a$ of length $j-1$ and $j+1$ respectively, exactly two subsets $\sigma$ of $a$ will satisfy $\sigma'\subset\sigma\subset\sigma''$; hence $b$ belongs to exactly two $(n-1)$-faces. It is easy to see that $\sd(\partial\Delta_n(a))$ is $(n-1)$-path connected; this follows from the fact that $\partial\Delta_n(a)$ is path-connected, as is the subdivision of any $(n-1)$-simplex, and in particular of any face of $\partial\Delta_n(a)$.
\end{proof}

For any $a\in [\omega_n]^{n+1}$ define the simplicial complex $X(a)$ on the vertex-set $F``[a]^{\leq n}$ by letting $\{F(\sigma_0),\dots,F(\sigma_j)\}\in X(a)$ whenever $\sigma_0\subset\dots\subset\sigma_j$.

\begin{claim} \label{claim2}
  There exists an $a\in [\omega_n]^{n+1}$ such that $X(a)$ contains as a subcomplex an $(n-1)$-cycle $t$.
\end{claim}

\begin{proof}
By Claim \ref{claim1},
\begin{align*}
d_{n-1}\Bigg(\sum_{b\in[\sd(\partial\Delta_n(a))]_{n-1}}\langle b\rangle\Bigg)=0.
\end{align*}
(Recall that our homology computations take coefficients in $\Zbb/2\Zbb$.) It follows that
\begin{align}
\label{eq:homological_sum}
d_{n-1}\Bigg(\sum_{b\in[\sd(\partial\Delta_n(a))]_{n-1}}\langle F``b\rangle\Bigg)=0.
\end{align}
Observe then that if some $F``b$ is the image of a \emph{unique} $b\in\sd(\partial\Delta_n(a))$ then the argument of $d_{n-1}$ in (\ref{eq:homological_sum}) is nontrivial and hence corresponds, by Proposition \ref{prop:connon}, to a family of $(n-1)$-cycles within $X(a)$; any $(n-1)$-path component of such a family is then an $(n-1)$-cycle $t$ such as we desire.

It suffices therefore to choose an $a\in [\omega_n]^{n+1}$ possessing such a $b\in a\bb{n}$.
We do so as follows.
Let $\alpha_0=0$.
If for some $j\leq n$ the ordinals $\alpha_i$ $(i<j)$ have all been defined, then let $\alpha_j=\max\{F(\sigma)\mid\sigma\in \big[ \{\alpha_0,\dots,\alpha_{j-1}\}\big]^{\leq j}\}+1$.
Let $a=\{\alpha_0,\dots,\alpha_n\}$.
Our procedure ensures that $$\Big|F^{-1}\big(F(\{\alpha_0,\dots,\alpha_j\})\big)\cap [a]^{\leq n+1}\Big|=1$$ for all $j\leq n$, and hence that $b=(\{\alpha_0\},\{\alpha_0,\alpha_1\},\dots,\{\alpha_0,\dots,\alpha_{n-1}\})$ is as desired.
\end{proof}

This concludes the proof: by Claim \ref{claim2}
together with Lemma \ref{lemma:coloring_cones}, for any $n\geq 0$ and strictly increasing $(n+1)$-cofinal $F$, the function $c_n\circ F^*$ is non-constant; in other words, $c_n$ witnesses the failure of $\PH_n(\omega_n)$.
\end{proof}

Observe that the $n=0$ instance of Theorem \ref{theorem:partition_failure_ordinals} is sharp, in the sense that $\PH_0(\omega_1)$---and, indeed, $\PH_0(\varepsilon)$ for any ordinal $\varepsilon$ of cofinality other than $\omega$---is a ZFC theorem.
Under large cardinal assumptions, the $n=1$ instance of Theorem \ref{theorem:partition_failure_ordinals} is sharp as well. Recall that an ideal $\Ical\subset \Pscr(\lambda)$ is \emph{$\kappa$-dense} if $\Pscr(\lambda)/\Ical$ has a dense subset of cardinality $\kappa$.
\begin{thm}
\label{theorem:foreman_ideal}
If there exists a uniform, countably complete, $\aleph_1$-dense ideal $\Ical$ on $\omega_2$ then $\PH_1(\omega_2)$ holds.
\end{thm}

\begin{proof}
Fix a coloring $c:\omega_2^2\to\omega$ and an $\Ical$ as in the premise of the theorem.
Using $\Ical$, we will define a strictly increasing $2$-cofinal $F$ such that $c\circ F^*$ is constant.
Since $\Ical$ is countably complete, we may begin by fixing for each $\alpha\in\omega_2$ an $\Ical$-positive $A_\alpha$ and $i_\alpha\in\omega$ such that $c(\alpha,\beta)=i_\alpha$ for all $\beta\in A_\alpha$.
Since $\Ical$ is $\aleph_1$-dense, there exists an unbounded $B\subseteq\omega_2$ and $\Ical$-positive $X\subseteq\omega_2$ such that $X \backslash A_\alpha \in\Ical$ for all $\alpha\in B$; by the pigeonhole principle, there then exists an $i\in\omega$ and unbounded $C\subseteq B$ such that $i_\alpha=i$ for all $\alpha\in C$.
For each $\alpha$ in $\omega_2$ let $F(\alpha)=\min\,C\backslash\alpha$, and for each $\alpha<\beta$ in $\omega_2$ let
$$F(\alpha,\beta)=\min\big(A_{F(\alpha)}\cap A_{F(\beta)}\backslash (F(\alpha)\cup F(\beta)+1)\big)$$
(it is the existence of the set $X$ that ensures that this expression is meaningful).
The composition $c\circ F^*$ takes the constant value $i$, as desired.
\end{proof}
\begin{rem}
Working from the assumption of a huge cardinal, Foreman constructs a uniform, countably complete, $\aleph_1$-dense ideal on $\omega_2$ in \cite{F98}.
It is not difficult to see that some large cardinal assumption is necessary for the conclusion of Theorem \ref{theorem:foreman_ideal}; by Proposition \ref{square} this assumption is at least that of a weakly compact cardinal.
We return to the question of the consistency strengths of $\PH_n(\omega_{n+1})$ in our conclusion below.
\end{rem}
Clearly the argument of Theorem \ref{theorem:foreman_ideal} will continue to apply with any $\kappa$ and $\kappa^+$ in the place of $\aleph_1$ and $\omega_2$, respectively.
A perhaps more interesting generalization of the argument arises with the question of $\PH_2(\omega_3)$.
Write $\add(\Ical)$ and $\dens(\Ical^+)$ for the completeness and density, respectively, of an ideal $\Ical$ on $\kappa$ (so that $\add (\Ical)=\dens (\Ical^+)=\aleph_1$ for the ideal of \cite{F98} invoked above, for example).
\begin{thm}
\label{theorem:omega_3_variant}
If there exist uniform ideals $\Ical$ and $\Jcal$ on a cardinal $\kappa$ satisfying
\begin{enumerate}

\item \label{I_sigma-add} $\aleph_1\leq\text{add}(\Ical)$,
\item \label{dens(I)<add(J)} $\text{dens}(\Ical^+)<\text{add}(\Jcal)\leq\kappa$, and
\item \label{dens(J)<kappa} $\text{dens}(\Jcal^+)<\kappa$,

\end{enumerate}
then $\PH_2(\kappa)$ holds.
\end{thm}
\begin{proof}
Fix a coloring $c:\kappa^3\to\omega$.
We will construct a $3$-cofinal $F$ such that $c\circ F^*$ is constant.
By premise (\ref{I_sigma-add}), for all $\alpha<\beta$ in $\kappa$ there exists an $i_{\alpha\beta}\in\omega$ with $A_{\alpha\beta}:=\{\gamma\in\kappa\mid c(\alpha,\beta,\gamma)=i_{\alpha\beta}\}\not\in\Ical$.
Fix a set $E$ of representatives for the elements of a dense subset of $\Pscr(\kappa)/\Ical$ such that $|E| = \dens(\Ical^+)$.
By premise (\ref{dens(I)<add(J)}),
for each $\alpha\in\kappa$ there exists a $\Jcal$-positive $B_\alpha\subseteq\kappa$ and $X_\alpha\in E$ and $i_\alpha\in\omega$ with $ X_\alpha \backslash A_{\alpha\beta}\in\Ical$ and $i_{\alpha\beta}=i_\alpha$ for all $\beta\in B_\alpha$.
There then exists an $X\subseteq\kappa$ and $i\in\omega$ and $\Jcal$-positive $C$ such that $X_\alpha=X$ and $i_\alpha=i$ for all $\alpha\in C$.
By premise (\ref{dens(J)<kappa}), there exists a $\Jcal$-positive $Y$ and unbounded $D\subseteq C$ such that $Y \backslash B_\alpha\in\Jcal$ for all $\alpha\in D$.
Now (partially) define $F:\kappa^{\leq 3}\to\kappa$ as follows:
\begin{itemize}
\item Let $F(\alpha)=\min(D\backslash\alpha)$ for all $\alpha\in\kappa$,
\item let $F(\alpha,\beta)=\min\big(B_{F(\alpha)}\cap B_{F(\beta)}\backslash(F(\alpha)\cup F(\beta)+1)\big)$ for all $\alpha<\beta$ in $\kappa$, and
\item let $$F(\alpha,\beta,\gamma)=\min\Bigg(\bigcap_{\emptyset\subset\sigma\subset\tau\subset\{\alpha,\beta,\gamma\}}A_{F(\sigma),F(\tau)}\Big\backslash\bigg(\bigcup _{\emptyset\subset\sigma\subset\{\alpha,\beta,\gamma\}} F(\sigma)+1\bigg)\Bigg)$$ for all $\alpha<\beta<\gamma$ in $\kappa$.
\end{itemize}
Again it is the existence of the sets $Y$ and $X$, respectively, that ensures that the second and third of these expressions is meaningful.
It is now straightforward to see that $c\circ F^*$ takes the constant value $i$, as desired.
\end{proof}

Intriguingly, for any accessible cardinal $\kappa$ the consistency of Theorem \ref{theorem:omega_3_variant}'s premises with the ZFC axioms is unknown.
More particularly, when $\kappa=\omega_3$ they entail the existence of an $\aleph_1$-dense ideal on $\omega_3$, the consistency of which is a well-known open question.
Note also that longer sequences of ideals satisfying premises like those of Theorem \ref{theorem:omega_3_variant} will allow for even higher-order versions of the above argument, securing the consistency of $\PH_n(\kappa)$ wherever the premises themselves are.

We close this section with a few observations reconnecting our ordinal analyses with our paper's main focus on partition hypotheses on $\Omega$.
\begin{thm}
\label{theorem:weakly_compact}
$\PH_n(\kappa)$ holds for any weakly compact cardinal $\kappa$ and $n\in\omega$.
\end{thm}

\begin{proof}
Assume $\kappa$ is weakly compact and  
suppose that $c:\kappa^n \to \omega$.
Let $m \in \omega$ and $X \subseteq \kappa$ be such that $|X| = \kappa$ and
whenever $\vec{\alpha} \in X^n$ and has strictly increasing coordinates,
$c(\vec{\alpha}) = k$.
Define $F:\kappa^{\leq n} \to X$ by letting $F(\xbf)$ be the least element of $X$
strictly larger than any coordinate of $\xbf$ or $F(\ybf)$ for any $\ybf \vartriangleleft \xbf$.
It is easily verified that $F$ is an $n$-cofinal function, that the range of $F^*$ consists of
strictly increasing $n$-tuples from $X$, and hence that $c \circ F^*$ takes the constant value $i$.
\end{proof}

The connection with our results on partition hypotheses on $\Omega$ is the following: the length-$\kappa$ iteration of Hechler forcings appearing in both \cite{SVHDL} and \cite{strong_hom_add} effectively translates the partition hypotheses of Theorem \ref{theorem:weakly_compact} to the setting of $\Omega$.
Those works' subsequent deductions are now encapsulated by Section \ref{Sect:PartHypAdd}.
See the conclusion of \cite{SVHDL} for the outline of an argument that any length-$\kappa$
finite-support iteration of $\sigma$-centered posets of cardinality less than $\kappa$ will achieve the same effect.
A more direct translation is the following. 
\begin{prop}
\label{Prop:7.19}
Suppose that $\mathfrak{b}=\mathfrak{d}=\kappa$ and $\PH_n(\kappa)$. 
Then $\PH_n$. 
\end{prop}
\begin{proof}
A $\kappa$-scale yields an $f\colon \kappa\to\Om$ which is increasing with cofinal range. 
Lemma \ref{reduction} completes the proof. 
\end{proof}
Partition hypotheses on $\Om$ also imply partition hypotheses on the ordinals; these interrelationships allow us to compute the exact consistency strengths of the principles $\PH_n$ for all $n\geq 0$.
To this end, we'll need one preliminary definition and proposition.
\begin{defn}
For any cardinal $\kappa$ the principle $\square(\kappa)$ is the assertion that there exists a sequence $\mathcal{C}=\langle C_\alpha\mid\alpha\in\kappa\rangle$ such that
\begin{itemize}
\item $C_\alpha$ is a closed unbounded subset of $\alpha$ for each $\alpha\in\kappa$. 
\item $C_\beta\cap \alpha=C_\alpha$ for every $\beta\in\kappa$ and limit point $\alpha$ of $C_\beta$. 
\item No club $C\subseteq\kappa$ satisfies $C\cap \alpha=C_\alpha$ at every limit point $\alpha$ of $C$. 
\end{itemize}
\end{defn}
Recall from \cite{To87} and \cite{Jensen} that if $\square(\kappa)$ fails at a regular uncountable
cardinal $\kappa$ then $\kappa$ is weakly compact in $L$.

\begin{prop} \label{square}
Suppose that $\kappa$ is a regular uncountable cardinal and that $\square(\kappa)$ holds. 
Then $\PH_n(\kappa)$ fails for all $n>0$. 
\end{prop}
\begin{proof}
Fix a $\square(\kappa)$-sequence  $\mathcal{C}=\langle C_\alpha\mid\alpha\in\kappa\rangle$. Let $\rho_2:[\kappa]^2\rightarrow \omega $ be the characteristic of walks along the sequence defined recursively by $\rho_2(\alpha, \beta)=\rho_2(\alpha, \min(C_\beta\setminus \alpha))+1$ with the boundary condition $\rho_2(\alpha,\alpha)=0$. We shall need the following two properties of this characteristic (see \cite[1.14]{To87}):
\begin{itemize}
\item $d(\alpha, \beta)=\sup_{\xi\leq \alpha}|\rho_2(\xi, \alpha)-\rho_2(\xi, \beta)|<\infty$ for all $\alpha<\beta<\kappa$.
\item for any $i\in\omega$ and $\mathcal{A},\mathcal{B}\in[\kappa]^\kappa$, there exist $\alpha\in \mathcal{A}$ and $\beta\in \mathcal{B}$ such that $\rho_2(\alpha, \beta)>i$.
\end{itemize}
Since $d(\alpha, \beta)\geq \rho_2(\alpha, \beta)$ for all $\alpha<\beta<\kappa$ it follows that the function 
$d\colon[\kappa]^2\to\omega$ has the following two properties:
\begin{itemize}
    \item  $d(\alpha,\beta)\leq d(\alpha,\gamma)+d(\beta,\gamma)$ for every $\alpha<\beta<\gamma<\kappa$;
    \item for any $i\in\omega$ and $\mathcal{A}\in[\kappa]^\kappa$, there are $\alpha,\beta\in\mathcal{A}$ such that $d(\alpha,\beta)>i$. 
\end{itemize}
As before, we may identify $d$ with a partial function $\kappa^2\to\omega$; let $c$ denote any total extension of this function. 
Let $F$ be a strictly increasing $2$-cofinal function and let $\mathcal{A}=\{F(\alpha)\mid\alpha\in\kappa\}$.
The set $\mathcal{A}$ is unbounded in the regular cardinal $\kappa$, hence $|\mathcal{A}|=\kappa$. 
By the above, for any $i\in\omega$ there exist $F(\alpha)<F(\beta)$ such that 
\[2i<c(F(\alpha),F(\beta))\leq c(F(\alpha),F(\alpha,\beta))+c(F(\beta),F(\alpha,\beta)).\]
In consequence, $c\circ F^*$ cannot take the constant value $i$. This shows the failure of $\mathrm{PH}_1(\kappa)$, and since $\mathrm{PH}_n(\kappa)$ implies $\mathrm{PH}_m(\kappa)$ for all $n\geq m\geq 0$, this completes the proof.
\end{proof}
Note that $\PH_n(\kappa)$ is equivalent to $\PH_n(\operatorname{cf}(\kappa))$.
\begin{prop}
\label{OmtoOrds}
The partition hypothesis $\PH_n$ implies $\PH_n(\mathfrak{d})$. In particular, it implies that $\operatorname{cf}(\mathfrak{d})>\omega_n$; if $n>0$, it implies moreover that $\operatorname{cf}(\mathfrak{d})$ is weakly compact in $L$.
\end{prop}
\begin{proof}
If $\langle x_\alpha\mid\alpha<\mathfrak{d}\rangle$ is cofinal with respect to $\leq^*$, then $x\mapsto\min\{\alpha\mid x\leq^*x_\alpha\}$ is monotone with cofinal image. 
Lemma \ref{reduction} completes the proof. 
As we have noted, if $\square(\operatorname{cf}(\mathfrak{d}))$ fails then $\operatorname{cf}(\mathfrak{d})$ is weakly compact in $L$, hence the conclusion for $n>0$ follows from Proposition \ref{square}. 
\end{proof}
\begin{cor}
For all $n>0$ the consistency strength of the principle $\PH_n$ is exactly a weakly compact cardinal.
\end{cor}
\begin{proof}
This follows immediately from Theorem \ref{PHfromHechlers} and Proposition \ref{OmtoOrds}.
\end{proof}
\subsection*{Partition relations and hypotheses from a simplicial perspective}
In this section we show that our partition hypotheses admit concise formulation within the framework of simplicial homotopy theory, and, moreover, that they figure therein as only very minor variations on classical partition relations.

To do so, we recall some basics from the theory of simplicial sets. For brevity, we will leave several terms incidental to our ultimate aim only very loosely defined; readers are referred to Chapters I and III.4 of \cite{GJ} for a much fuller treatment.

\begin{defn}
Write $\mathbf{\Delta}$ for the category of finite nonempty ordinals, whose objects are typically written $[0]=\{0\}$, $[1]=\{0,1\}$, and so on, and whose morphisms are the order-preserving maps $f:[m]\to [n]$.
Among these maps we distinguish two main sorts:
\begin{itemize}
\item injections of the form $d^i_n:[n-1]\to [n]$, which omits $i\in [n]$ from its image, and
\item surjections of the form $s^i_n:[n+1]\to [n]$, which sends $i\in [n]$ and $i+1$ both to $i$.
\end{itemize} 
It is easy to see that these maps generate the morphisms of $\mathbf{\Delta}$.
In consequence, the \emph{cosimplicial identities}, a short list of the fundamental relations among them (like $s_n^i d_{n+1}^i=\mathrm{id}$), fully determines the category $\mathbf{\Delta}$; similarly for its opposite category $\mathbf{\Delta}^{\mathrm{op}}$, wherein the \emph{simplicial identities} play the analogous role (\cite[p. 4]{GJ}).
As is standard, we will suppress such maps' indices $n$ in what follows.

A \emph{simplicial set} $X$ is a functor $S:\mathbf{\Delta}^{\mathrm{op}}\to\mathsf{Set}$; put differently, it is:
\begin{itemize}
\item a family of sets $F([n])=X_n$ $(n\in\omega)$, together with
\item morphisms $F(f):X_n\to X_m$ whose relations mirror those among the morphisms in $\mathbf{\Delta}$.
\end{itemize}
The elements of $X_n$ are sometimes termed the \emph{$n$-dimensional faces} of $X$.
Maps of the form $F(d^i)$ or $F(s^i)$ are termed \emph{face maps} and \emph{degeneracies} and written $d_i$ and $s_i$, respectively.
An element of the image of the latter is termed \emph{degenerate}.
We write $\mathsf{sSet}$ for the category of simplicial sets; the morphisms therein are the natural transformations.
\end{defn}

A shaping intuition for simplicial sets is the following example.

\begin{exa}
\label{ex:simplicialcomplexesassimplicialsets}
Let $Z$ be totally ordered and let $Y$ be an abstract simplicial complex on $Z$. 
The order on $Z$ induces a family of face maps $d_i$ on the sequence $\mathcal{Y}=\{[Y]_n\mid n\in\omega\}$ which satisfy the relevant simplicial identities; lacking degeneracies, however, $\mathcal{Y}$ fails to define a simplicial set.
Nevertheless, there exists a minimal simplicial set $\overline{Y}$ which (levelwise) contains $\mathcal{Y}$, and it amounts simply to the closure of $\mathcal{Y}$ under the degeneracy maps $s_i$.
\end{exa}
Like abstract simplicial complexes, simplicial sets admit geometric realizations --- in fact there exists a geometric realization functor $|\cdot |:\mathsf{sSet}\to\mathsf{Top}$ which is left adjoint to the \emph{singular functor} $\mathrm{Sing}:\mathsf{Top}\to\mathsf{sSet}$,
where $\mathrm{Sing}(Y)_n$ is just the set of continuous maps from an $n$-simplex to the topological space $Y$ (this is, of course, the functor underlying the singular homology of $Y$).
The remarkable point is that whatever distortions or identifications these two functors may introduce, they do respect the fundamental notions of homotopy (fibrations, cofibrations, weak equivalences) on each side \cite{Qui}.
Moreover, each of these functors has image within that class of objects best suited for homotopy operations: the \emph{$CW$ complexes} in $\mathsf{Top}$; the \emph{Kan complexes} in $\mathsf{sSet}$.

For our purposes, the crucial example of a simplicial set is the following one.
\begin{exa}
For any quasi-order $\mathbb{P}$, let $N\mathbb{P}$ denote the \emph{nerve of $\mathbb{P}$}; this is the simplicial set defined by $(N\mathbb{P})_n=\mathbb{P}^{[n+1]}$ for all $n\in\omega$, with the morphism $(N\mathbb{P})_n\to (N\mathbb{P})_m$ associated to any morphism $[m]\to [n]$ in $\mathbf{\Delta}$ the obvious one.
\end{exa}
\begin{rem}
\label{rem:simplicialcolorings}
Observe that in the perspective of the previous example, a function $c:[\kappa]^{n+1}\to\lambda$ is just a $\lambda$-coloring of the nondegenerate $n$-faces of $N\kappa$.
\end{rem} 
The following theorem is folklore; see \cite[\href{https://kerodon.net/tag/002Z}{Tag 002Z}]{Ker} for a proof.
\begin{thm}
\label{thm:folklore}
The nerve functor $N$ fully and faithfully embeds the category $\mathsf{QO}$ of quasi-orders into the category $\mathsf{sSet}$ of simplicial sets.
\end{thm}
Within the framework of $\mathsf{sSet}$, the functor $N$ has one main shortcoming, which is that the nerve of a (nontrivial) quasi-order is never a Kan complex.
More important than Kan complexes' definition for our immediate purposes is their analogy with CW complexes: just as any topological space may be replaced by a weakly equivalent CW complex, any simplicial set may be converted to a Kan complex via repeated applications of the \emph{$Ex$ functor}, which should be thought of as the \emph{reverse} (or more precisely right adjoint) \emph{of the subdivision functor on simplicial sets} (see \cite[p. 183]{GJ}).
We will define its levels after the following example, then proceed directly to this section's main aim: the parallel reformulations of classical partition relations and the hypotheses $\mathrm{PH}_n$.
\begin{exa}
Recall from Definition \ref{simplicial_definitions} the subdivision $\mathrm{sd}(Y)$ of an abstract simplicial complex $Y$.
Note also that a partial order linearly ordering the vertices of each of the faces of $Y$ would have sufficed for the construction of Example \ref{ex:simplicialcomplexesassimplicialsets}.
Hence $\overline{\mathrm{sd}(Y)}$ is well-defined, since $Z=(Y,\subseteq)$ forms such an ordering of the vertices of $\mathrm{sd}(Y)$; in fact it's straightforward to see that $\overline{\mathrm{sd}(Y)}\cong NZ$.

Write $\Delta^n$ for the simplicial set $\overline{\Delta_n}$, where $\Delta_n$ is, much as in Definition \ref{simplicial_definitions}, the abstract simplicial complex on $\mathbb{N}$ corresponding to the standard $n$-simplex; this is equivalent to the more standard definition $\Delta^n=\mathrm{Hom}_{\mathbf{\Delta}}(\,\cdot\,,[n])$.
Define then the subdivision $\mathrm{sd}$ of the \emph{simplicial set} $\Delta^n$ by $\mathrm{sd}\,\Delta^n:=\overline{\mathrm{sd}(\Delta_n)}\cong N(\Delta_n,\subseteq)$.
\end{exa}
\begin{defn}
For any simplicial set $X$, the levels of the $Ex$-image of $X$ are defined by
$$(Ex\,X)_n=\mathrm{Hom}_{\mathsf{sSet}}(\mathrm{sd}\,\Delta^n,X)$$
for all $n\in\omega$.
\end{defn}
\begin{lem}
\label{lem:Exofanerve}
For any quasi-order $\mathbb{P}$, the elements of $(Ex\,N\mathbb{P})_n$ are exactly the order-preserving images of $(\Delta_n,\subseteq)$ in $\mathbb{P}$.
\end{lem}
\begin{proof}
By definition, $(Ex\,N\mathbb{P})_n=\mathrm{Hom}_{\mathsf{sSet}}(\mathrm{sd}\,\Delta^n,N\mathbb{P})=\mathrm{Hom}_{\mathsf{sSet}}(N(\Delta_n,\subseteq),N\mathbb{P})$.
By Theorem \ref{thm:folklore}, this in turn equals $\mathrm{Hom}_{\mathsf{QO}}((\Delta_n,\subseteq),\mathbb{P})$, completing the proof.
\end{proof}
We require only one further item of notation.
\begin{notn}
For any $n\geq 0$ the family of maps $v_i:[0]\to[n]:0\mapsto i$ determines a family of maps $v_i^*:X_n\to X_0$. Let $$\mathrm{vert}(x)=\{v_i^*(x)\mid i\in [n]\}$$ for any $x\in X_n$.
\end{notn}
\begin{defn} For any simplicial set $X$ and $n>0$, say $S\subseteq X_n$ \emph{spans} $T\subseteq X_0$ if for every $\bar{t}\in [T]^{n+1}$ there exists an $s\in S$ with $\mathrm{vert}(s)=\bar{t}$.
Say $S$ \emph{spans} $T$ \emph{neatly} if for each $i\leq n$ the vertex map is injective on the collection $\{d_i(s)\mid s\in S\}$ of $i$-faces of elements of $S$.
\end{defn}

We now have the following equivalences. For any infinite cardinal $\kappa$:
\begin{itemize}
\item $\kappa\to(\kappa)_\omega^{n+1}$ is equivalent to:
\medskip

\begin{quote}
\emph{For all $c:[\kappa]^{n+1}\to\omega$ there exists a cofinal $T\subseteq\kappa$ and a $c$-monochromatic $S\subseteq (N\kappa)_n$ neatly spanning $S$.}
\end{quote}
\medskip

\item $\mathrm{PH}_n(\kappa)$ is equivalent to:
\medskip
\begin{quote}
\emph{For all $c:\kappa^{n+1}\to\omega$ there exists a cofinal $T\subseteq \kappa$ and a $c$-monochromatic $S\subseteq (Ex\,N \kappa)_n$ neatly spanning $T$.}
\end{quote}
\medskip

\item More generally, $\mathrm{PH}_n(\mathbb{P},\lambda)$ asserts for any quasi-order $\mathbb{P}$ and cardinal $\lambda$ that:

\medskip
\begin{quote}
\emph{For all $c:\mathbb{P}^{n+1}\to\lambda$ there exists a cofinal $T\subseteq \mathbb{P}$ and a $c$-monochromatic $S\subseteq (Ex\,N \mathbb{P})_n$ neatly spanning $T$.}
\end{quote}
\medskip

\end{itemize}
The first item follows directly from Remark \ref{rem:simplicialcolorings}, and the second and third are similarly immediate from definitions and Lemma \ref{lem:Exofanerve}.
In the second item, for example, note that the requirement that $S$ \emph{neatly} spans $T$ ensures that the association to each $\bar{t}\in [T]^{k+1}$ $(k\leq n)$ of an order-preserving image of $(\Delta_k,\subseteq)$ in $\kappa$ (one identifying the vertices of $\Delta_k$ with the elements of $\bar{t}$) is well-defined. This $\subseteq$-increasing association of values in $\kappa$ to, for example, the nonempty subsets $\{\alpha\}$, $\{\beta\}$, and $\{\alpha,\beta\}$ of each $\{\alpha,\beta\}\in [T]^2$ amounts to a partial $2$-cofinal function $T^{\leq 2}\to \kappa$, one which will monochromatically extend to a full $2$-cofinal function $\kappa^{\leq 2}\to\kappa$ by way of the remark following Lemma \ref{strictly_increasing} and (the logic of) Lemma \ref{extensionlemma}.
Further details are left to the reader.
The fact that in the simplicial language outlined above, the formulations of $\mathrm{PH}_n(\kappa)$ and of classical partition principles differ by only four characters is quite striking.

\section{Conclusion}
We conclude this work by recording several questions arising out of the analyses above.
First, we do not know whether ($\dagger$) is needed in the results in Sections \ref{Sect:topology} and \ref{Sect:MeasPartHyp}.
In particular:

\begin{question}
Are Theorem \ref{partition} or Corollary \ref{cor:CA_cohere->PCA_triv} true if the hypothesis of \emph{($\dagger$)} is removed?
\end{question}

It is also unclear whether the measurable cardinal is needed as a hypothesis in Proposition \ref{Hn_coll}.

\begin{question}
Does $L(\Rbb)$ satisfy that every subset of $\Omega^{[n]}$ is $\Hcal_n$-measurable 
after Levy collapsing an inaccessible cardinal to $\omega_1$?
\end{question}

\begin{question}
Is ($\dagger$) equivalent the assertion that $\aleph_1$ is an inaccessible cardinal in $L[r]$ for every $r \subseteq \omega$?
\end{question}
\noindent
This seems closely related to whether ZFC proves $\Qcal_n$ is c.c.c..

Two other conspicuous questions are the following:
\begin{question}
What is the consistency strength of $\PH_1(\omega_2)$?
\end{question}

\begin{question}
What is the consistency strength of $\PH_2(\omega_3)$?
\end{question}
\noindent
Here we should acknowledge the possibility of the answer ``$0=1$''; in other words, the possibility exists that $\PH_2(\omega_3)$ inconsistent with the ZFC axioms.

Returning to the setting of $\Omega$:

\begin{question}
What is the least value of the continuum compatible with the condition ``$\PH_n$ for all $n\in\omega$''?
\end{question}
\noindent
By Proposition \ref{OmtoOrds}, $\aleph_{\omega+1}$ is a lower bound.
We note that in \cite{SVHDLwLC} it is established that the vanishing of
${\lim}^n \Abf$ for all $n$---a consequence of $\PH_n$---is relatively consistent with
ZFC via a model in which the continuum is $\aleph_{\omega+1}$.
The question of whether this is optimal depends on answers to the following:
\begin{question}
If $n > 1$, what is the least value of the continuum compatible with ${\lim}^s \Abf = 0$ for all $s \leq n$?
What if $\mathfrak{b} = \dfrak$?
\end{question}
\noindent
The $n=1$ case was settled in \cite{DSV}, which established the consistency of ${\lim}^1 \Abf = 0$ with $\mathfrak{b} = \mathfrak{d} = 2^{\aleph_0} = \aleph_2$. Note that upper bounds $\lambda_n$ $(n>0)$ for the first part of the question are recorded in \cite[Theorem 6.1]{SVHDLwLC} (under the Generalized Continuum Hypothesis $\lambda_1=\aleph_2$, $\lambda_2=\aleph_7$, and so on, in a sequence with supremum $\aleph_\omega$).

As the preceding discussion underscores, the principles $\PH_n$ are stronger than is strictly necessary for the conclusions about higher limits that we derive from them.
More precisely, while $\PH_n$ quantifies over all colorings $c:\Omega^{n+1}\to\omega$, only a subclass of these colorings pertain to additivity questions for ${\lim}^n$.
Useful but more attainable variants of $\PH_n$ might amount to restrictions to just this class.
Put differently, while we have shown the consistency strength of $\PH_n$ to be equal to existence of a weakly compact cardinal, it seems likely that the consistency strength of the additivity of the associated $\lim^n$ functors is significantly less. 

\begin{question}
What is the consistency strength of the statement ``$\lim^n$ is additive on the class of $\Om$-systems''?
\end{question}

A related calibration of strength is the following question.

\begin{question} 
Suppose ${\lim}^n$ is additive for $\Omega$-systems in the inner model $L(\Rbb)$. Can we conclude that $\omega_1$ is an inaccessible cardinal in $L$?
\end{question}

Since large cardinals imply that $L(\Rbb)$ models that ${\lim}^n$ is additive for $\Omega$-systems,
it is natural to ask if this conclusion can be derived from a determinacy hypothesis.
\begin{question}
Assume the Axiom of Determinacy. 
Is ${\lim}^n$ additive for $\Omega$-systems?
\end{question}

We noted a further question in our discussion of generalized partition hypotheses in Section \ref{Sect:Generalizing}; this was the following:
\begin{question}
Let $\mathbb{P}$ and $\mathbb{Q}$ be directed quasi-orders with $\mathbb{P}\leq_T\mathbb{Q}$. 
Does $\PH_n(\mathbb{Q},\lambda)$ imply $\PH_n(\mathbb{P},\lambda)$?
\end{question}
This question is, in spirit, the converse of Lemma \ref{reduction}; by that lemma, in fact, this question is equivalent to that of whether $\PH_n(\mathbb{Q},\lambda)$ implies $\PH_n(\mathbb{P},\lambda)$ whenever $\mathbb{P}$ is a cofinal suborder of $\mathbb{Q}$.

Finally, it is natural to ask if results like ours on $L(\Rbb)$ can augment and refine the treatment and analysis of mathematical obstructions which derived limits tend to organize.
For example, Sections \ref{Sect:topology} and \ref{Sect:MeasPartHyp} may be read as evoking a subcomplex of the cochain complex of Definition \ref{limscochaincomplex}, one consisting only of its \emph{measurable} cochains.
Its cohomology groups might then be viewed as measurable variants of the functors ${\lim}^n$; by the results of Section \ref{Sect:MeasPartHyp}, these \emph{measurable higher limits} should be additive, and should even tend to equal zero.
The question is how algebraic a formalization any of this admits: whether such a family of functors forms a \emph{connected sequence of functors} in the sense of \cite{CaEi}, or whether, more particularly, they might correspond to a projective class in a category like $\mathsf{inv}$-$({^\omega}\omega)$, in the sense of \cite{EiMo2}.

\begin{question}
\label{ques:def_lim}
Do there exist measurable variants of the ${\lim}^n$ functors for $\Omega$-systems?
\end{question}
This question is somewhat open-ended; if the answer is, as suspected, yes, the question should be read as standing for the task of analyzing such variants and their relation to the classical functors ${\lim}^n$.
Applications in this case are not difficult to imagine; the existence of measurable---and by our results, better-behaved---variants of more compound functors like strong homology, for example, would be likely to follow.


\begin{thebibliography}{10}

\bibitem{strong_hom_add} N.~Bannister, J.~Bergfalk, J.~Tatch Moore. {\em On the additivity of strong homology
for locally compact separable metric spaces}, arXiv:2008.13089, February 22, 2021, 13pp.  

\bibitem{B17} J.~Bergfalk. {\em Strong homology, derived limits, and set theory}, Fund. Math. 236 (2017), no. 1, 17--28.

\bibitem{SVHDL} J.~Bergfalk, C.~Lambie-Hanson. {\em Simultaneously vanishing higher derived limits}, 
Forum Math, Pi 9 (2021), Paper no. e4, 31pp.

\bibitem{SVHDLwLC} J.~Bergfalk, M.~Hru\v{s}\'{a}k, C.~Lambie-Hanson. {\em Simultaneously vanishing higher derived limits without large cardinals}, arXiv:2102.06699, February 12, 2021, 30 pp.

\bibitem{BK} A.~K.~Bousfield, D.~M.~Kan. {\em Homotopy limits, completions and localizations}. Lecture Notes in Mathematics, Vol. 304. Springer-Verlag, Berlin-New York, 1972. v+348 pp.

\bibitem{BL} J.~Brendle, B.~L\"owe, {\em Solovay-type characterizations for forcing algebras},
J. Symbolic Logic 64 (1999) no. 3, 1307--1323.

\bibitem{CaEi} H.~Cartan, S.~Eilenberg, {\em Homological algebra}. With an appendix by David A. Buchsbaum. Princeton University Press, Princeton, NJ, 1956. xvi+390 pp.

\bibitem{Connon} E.~Connon. {\em On $d$-dimensional cycles and the vanishing of simplicial homology}, arxiv:1211.7087, July 22, 2013, 19 pp.

\bibitem{DSV} A.~Dow, P.~Simon, J.~Vaughan. {\em Strong homology and the proper forcing
axiom}, Proc. Amer. Math. Soc. 106.3 (1989), 821–-828.

\bibitem{UB} Q.~Feng, M.~Magidor, H.~Woodin. {\em Universally Baire sets of reals}, in 
Set theory of the continuum (Berkeley, CA, 1989), 203--242, Math. Sci. Res. Inst. Publ., 26,
Springer, New York, 1992.

\bibitem{EiMo2} S.~Eilenberg, J.~C.~Moore. {\em Foundations of relative homological algebra}.  Mem. Amer. Math. Soc. 55 (1965), 39 pp.

\bibitem{EiMo} \bysame. {\em Limits and spectral sequences}. Topology 1 (1962), 1--23.

\bibitem{F98} M.~Foreman. {\em An $\aleph_1$-dense ideal on $\aleph_2$}, Israel J. Math. 108 (1998), 253--290.

\bibitem{Gob} R\'{e}mi Goblot. {\em Sur les d\'{e}riv\'{e}s de certaines limites projectives. Applications aux modules}, Bull. Sci. Math. (2) 94 (1970), 251--255.

\bibitem{GJ} P.~Goerss, J.~Jardine. {\em Simplicial Homotopy Theory}. Birkh\"{a}user Verlag, Basel, 1999. xv+510 pp.

\bibitem{Jech} T.~Jech. {\em Set theory. The third millennium edition, revised and expanded.} Springer Monographs in Mathematics. Springer-Verlag, Berlin, 2003. xiv+769 pp.

\bibitem{Jensen} R. B. Jensen. {\em The fine structure of the constructible hierarchy}, Ann. Math. Logic 4 (1972), 229--308.

\bibitem{higher_infinite} A.~Kanamori. {\em The higher infinite. Large cardinals in set theory from their beginnings.}
Perspectives in Mathematical Logic. Springer-Verlag, Berlin, 1994. xxiv+536pp.

\bibitem{kechris} A.~Kechris. {\em Classical descriptive set theory.} Graduate Texts in Mathematics, 156. Springer-Verlag, New York, 1995. xviii+402 pp.

\bibitem{set_theory:kunen} K.~Kunen. {\em Set theory. An introduction to independence proofs.} Studies in Logic and the Foundations of Mathematics, 102. North-Holland Publishing Co., Amsterdam-New York, 1980. xvi+313 pp.

\bibitem{LHZ} C.~Lambie-Hanson, Z.~Zucker. {\em Polish space partition principles and the Halpern-L\"{a}uchli Theorem.} arXiv:2209.04859, September 11, 2022, 18 pp.

\bibitem{LevySolovay} A.~Levy, R.~Solovay. {\em Measurable cardinals and the continuum hypothesis}, Israel Journal of Mathematics 5 (1967), 234--248, 

\bibitem{Ker} J.~Lurie. {\em Kerodon}. Online 2018 text: \url{https://kerodon.net}, accessed Jan 2023.

\bibitem{MP} S.~Marde\v{s}i\'{c} and A.~V.~Prasolov. {\em Strong homology is not additive}, Trans. Amer. Math. Soc. 307.2 (1988), 725--744.

\bibitem{MartinSolovay} D.~A.~Martin, R.~Solovay. {\em A basis theorem for $\Sigma^1_3$ sets of reals}, Ann. of Math. 89 (1969), 138--159.

\bibitem{MaPo} J.~P.~May, K.~Ponto. {\em More concise algebraic topology. Localization, completion, and model categories}. Chicago Lectures in Mathematics. University of Chicago Press, Chicago, IL, 2012. xxviii+514 pp.

\bibitem{McG} C.~A.~McGibbon. {\em Phantom maps}, Handbook of algebraic topology, 1209–-1257, North-Holland, Amsterdam, 1995.

\bibitem{Mil} J.~Milnor. {\em On axiomatic homology theory}.  Pacific J. Math. 12 (1962), 337--341.

\bibitem{Mitch} B.~Mitchell. {\em Rings with several objects}, Advances in Math. 8 (1972), 1--161.

\bibitem{Nob}  G.~N\"{o}beling. \emph{\"{U}ber die Derivierten des Inversen und des direkten Limes einer Modulfamilie}, Topology 1 (1962), 47--61.

\bibitem{Oso} B.~L.~Osofsky. {\em The subscript of $\aleph_n$, projective dimension, and the vanishing of ${\lim}^n$}, Bull. Amer. Math. Soc. 80 (1974), 8--26.

\bibitem{Qui} D.~Quillen. {\em Homotopical algebra.} Lecture Notes in Mathematics 43. Springer, Berlin, 1967.

\bibitem{Roos} J.~E.~Roos. {\em Sur les foncteurs d\'{e}riv\'{e}s des lim.} C.~R.~Acad. Sci. Paris 252 (1961), 3702--3704.

\bibitem{SSH} S.~Marde\v{s}i\'{c}. {\em Strong shape and homology.}
Springer Monographs in Mathematics. Springer-Verlag, Berlin, 2000. xii+489 pp.

\bibitem{Pra} A.~V.~Prasolov. {\em Non-additivity of strong homology}, Topology Appl. 153 (2005), 493--527.

\bibitem{Sho} J.~R.~Shoenfield. {\em The problem of predicativity.}
1961 Essays on the foundations of mathematics pp. 132--139 Magnes Press, Hebrew Univ., Jerusalem.

\bibitem{SoTo}
S.~Solecki, S.~Todorcevic. {\em Cofinal types of topological directed
orders.} Ann. Inst. Fourier (Grenoble) 54 (2004), no. 6, 1877--1911.

\bibitem{So70} R.~M.~Solovay. {\em A model of set-theory in which every set of reals is Lebesgue measurable},
Ann. of Math. (2) 92 (1970), 1--56.

\bibitem{Ste} N.~E.~Steenrod. {\em Regular cycles of compact metric spaces}. Ann. of Math. 41 (1940), 833--851.

\bibitem{To98} S.~Todorcevic. {\em The first derived limit and compactly $F_\sigma$ sets}, J. Math. Soc. Japan 50 (1998), no. 4, 831–836.

\bibitem{To87} \bysame. {\em Partitioning pairs of countable ordinals}, Acta Math. 159 (1987), 261--294.

\bibitem{To89} \bysame. {\em Partition problems in topology.}
Contemporary Mathematics, 84. American Mathematical Society, Providence, RI, 1989. xii+116 pp. 

\bibitem{Wei}  C.~A.~Weibel. {\em An introduction to homological algebra}. Cambridge Studies in Advanced Mathematics, 38. Cambridge University Press, Cambridge, 1994. xiv+450 pp.

\bibitem{W88} W.~H.~Woodin. {\em Supercompact cardinals, sets of reals, and weakly homogeneous trees},
Proc. Nat. Acad. Sci. U.S.A. 85 (1988), no. 18, 6587--6591.

\bibitem{Omega-conjecture} \bysame. {\em The $\Omega$ conjecture},
in Aspects of complexity (Kaikoura, 2000), De Gruyter Ser. Log. Appl., no. 4, 155--169, de Gruyter, Berlin, 2001.

\bibitem{Yeh}  Z.~Z.~Yeh. {\em Higher inverse limits and homology theories.} Thesis (Ph.D.), Princeton University, 1959. 73 pp.

\bibitem{forcing_idealized} J.~Zapletal. {\em Forcing Idealized.} Cambridge Tracts in Mathematics, 174. Cambridge University Press, Cambridge, 2008. vi+314 pp.

\end{thebibliography}
\end{document}